\documentclass[a4paper,11pt]{article}

\usepackage[american]{babel} 
\usepackage[utf8]{inputenc}
\usepackage[a4paper, left=3cm, right=3cm, top=2.5cm, bottom=2.5cm]{geometry}
\usepackage{amssymb,amsmath,amsthm}
\usepackage{mathrsfs}
\usepackage{dsfont}

\usepackage{graphicx}
\usepackage{color}
\usepackage{tikz}

\numberwithin{equation}{section} % Adjusts Numbering of Equations
\usepackage{comment}
\usepackage{hyperref}
\newcommand{\eps}{{\varepsilon}}

\newcommand{\NN}{\nonumber}

\newcommand{\N}{\mathds{N}}

\newcommand{\R}{\mathds{R}}
\newcommand{\C}{\mathds{C}}
\newcommand{\K}{\mathds{K}}

\DeclareMathOperator{\re}{Re}

\DeclareMathOperator{\Tr}{Tr}

\DeclareMathOperator{\Cov}{Cov}

\DeclareMathOperator{\Id}{Id}
\DeclareMathOperator{\diag}{diag}
\DeclareMathOperator{\supp}{supp}
\DeclareMathOperator{\Spec}{Spec}

\DeclareMathOperator{\dist}{dist}

\renewcommand{\P}{\mathds{P}}
\newcommand{\E}{\mathds{E}}

\newcommand{\I}{\mathds{1}}

\renewcommand{\epsilon}{\varepsilon}

\renewcommand{\leq}{\leqslant}
\renewcommand{\geq}{\geqslant}
\renewcommand{\le}{\leq}
\renewcommand{\ge}{\geq}

\newcommand{\dx}{{\mathrm d}}

\newcommand{\rate}{\mathcal{I}}

\newtheorem{theorem}{Theorem}[section]         
\newtheorem{lemma}[theorem]{Lemma}             
\newtheorem{corollary}[theorem]{Corollary}

\newtheorem{assumption}[theorem]{Assumption}

\newtheoremstyle{myrem}%                % Name
  {}%                                     % Space above
  {}%                                     % Space below
  {}%                                     % Body font
  {}%                                     % Indent amount
  {\bfseries}%                            % Theorem head font
  {.}%                                    % Punctuation after theorem head
  { }%                                    % Space after theorem head, ' ', or \newline
  {\thmname{#1}\thmnumber{ #2}\thmnote{\normalfont{ (#3)}}}
\theoremstyle{myrem}
\newtheorem*{remark}{Remark}

\begin{document}
\title{\vspace{-2cm} Large deviations for the largest eigenvalue of a polynomial in Gaussian matrices}
\author{Alice Guionnet\thanks{UMPA, ENS de Lyon, France. E-mail: alice.guionnet@ens-lyon.fr.} \and Jonathan Husson\thanks{LMBP, Université Clermont Auvergne, France. E-mail: jonathan.husson@uca.fr.} \and Jana Reker\thanks{ICJ, Université Lyon 1, France. E-mail: reker@math.univ-lyon1.fr.}}

\maketitle

\begin{abstract}In this article, we prove that the spectrum, the norm, and the largest eigenvalue of a self-adjoint polynomial $P$ in $k$ independent matrices from the GOE or the GUE satisfy a large deviation principle.
\end{abstract}

Voiculescu Entropy.

\section{Introduction}
Large deviations theory for random matrices was initially motivated by Voiculescu’s work, which sought to develop an entropy theory for non-commutative variables as a framework for studying von Neumann algebras. Here, the concept of microstates entropy was first introduced by asymptotically measuring the volume of matrices whose polynomial traces approximate those of operators in a von Neumann algebra \cite{Voiculescu93, Voiculescu94}. This idea was later extended to a topological entropy by replacing the trace with the operator norm, effectively measuring the asymptotic volume of matrices for which the norms of polynomials in Gaussian matrices closely match those of given operators. However, except in the case of a single matrix, these entropy theories remain incomplete. Many expected properties have yet to be proven, primarily because a full large deviation principle (LDP) has not been established. As a result, entropies can be defined using either a limsup or liminf approach, both definitions being eventually different. Properties of the entropies such as summability under freeness are then unclear. 
In \cite{Austin-split2}, Austin introduced a notion of entropy analogous to sofic entropy for sequences of unitary representations of separable $C^*$-algebras. He defined an annealed almost periodic (AP) entropy by measuring the volume of polynomials in independent unitary matrices whose scalar evaluations—along a finite set of vectors—approximate those of a given completely positive map. The term annealed reflects the fact that the vectors are chosen randomly, and one evaluates the volume after averaging with respect to their distribution. Additionally, Austin introduced a zeroth-order AP entropy, which characterizes the probability that a given completely positive map arises as the positive functional associated with a representation via some vector in that representation. One of the key consequences of this new entropy theory is a large deviation principle for the largest eigenvalue of a polynomial in unitary matrices under the Haar measure on the unitary group~\cite[Cor.~14.4]{Austin-split2}.

\medskip
In this article, we prove that the spectrum, the largest eigenvalue, and the norm of a self-adjoint polynomial $P$ in~$k$ independent matrices drawn from the Gaussian Orthogonal Ensemble (GOE) or the Gaussian Unitary Ensemble (GUE) satisfy a large deviation principle. The results extend to the case where we consider the joint large deviations of the spectrum of several polynomials. Note that our result differs from~\cite[Cor.~14.4]{Austin-split2} as the underlying random matrix model is different: we consider Gaussian matrices rather than unitaries. It may be possible to deduce it from an extension of \cite[Cor.~14.4]{Austin-split2} that would include deterministic matrices, by diagonalizing the Gaussian matrices. This, however, seems quite a non-trivial extension. Moreover, our approach does not require applying the contraction principle to deduce the desired large deviation principle from a more complicated object. As a result, we obtain shorter proofs and a quite concrete rate function.

\medskip
Our strategy follows the idea of tilting by spherical integrals, which was originally introduced in \cite{GuionnetHusson2020} to establish the large deviations for the largest eigenvalue of Wigner matrices with sharp sub-Gaussian entries such as uniform or Rademacher variables. This argument was refined in \cite{AGH2021, CookDG23} to include Wigner matrices with rather general sub-Gaussian entries. Moreover, this new theory was applied to obtain the large deviation principles of the largest eigenvalue of more structured matrices such as covariance matrices \cite{Husson2022, HussonMcKenna2023, DGH2024}, Kronecker matrices \cite{LDP_Kronecker} or sums of a deterministic matrix and a Gaussian matrix \cite{Mckenna21,BoursierG24}. Our approach combines ideas from \cite{LDP_Kronecker,HussonMcKenna2023} and the linearization trick in~\cite{HaagerupTobjornsen2005, HST2006}. 

\medskip
\textbf{Acknowledgements:} A.G. and J.R. were partially supported by ERC Advanced Grant ``LDRaM'' No. 884584. An LLM was used to verify and correct the covariance calculation in Lemma \ref{lem-annealedcalc}, and its generalization to the GUE case in Lemma \ref{lem-annealedcalc-GUE}. The rest of the article is not affected by these changes.

\subsection{Assumptions on the Model}
We consider a $N\times N$ random matrix of the form
\begin{displaymath}
P_{\beta,N}:=P(W_1^{(N)},\dots,W_k^{(N)})\, ,
\end{displaymath}
where $P$ is a non-commutative polynomial in $k$ variables independent of $N$, and $W_1^{(N)},\dots,W_k^{(N)}$ are i.i.d.  random matrices. For simplicity, we usually drop the $N$-dependence of the random matrices unless we want to emphasize it specifically. The parameter $\beta\in\{1,2\}$ distinguishes between a real ($\beta=1$) and a complex ($\beta=2$) variant of the model, for which we make the following assumptions.

\begin{assumption}\label{assum-main}
Fix $\beta\in\{1,2\}$.
\begin{itemize}
\item[$\beta=1$] The matrices $W_1,\dots, W_k$ are sampled independently from the Gaussian Orthogonal Ensemble (GOE), i.e., $W_1=W_1^T$ and $(W_1)_{ij}$, $i\leq j$, are independent real centered Gaussian random variables with variance $(1+\I_{\{i=j\}})/N$. Moreover, we assume the coefficients of the polynomial $P$ to be $N$-independent real constants, and $P(X_1,\dots,X_k)$ to be symmetric if its arguments $X_1,\dots,X_k$ are.
\item[$\beta=2$] The matrices $W_1,\dots,W_k$ are sampled independently from the Gaussian Unitary Ensemble (GUE), i.e., $W_1=W_1^*$, where $(W_1)_{ij}$, $i<j$, are independent complex centered Gaussian random variables with variance $1/N$ (meaning that their real and imaginary parts are independent Gaussians of variance $1/(2N)$), and $(W_1)_{jj}$ are independent real centered Gaussian random variables with variance $1/N$ that are further independent from $(W_1)_{ij}$, $i<j$. Moreover, we assume the coefficients of the polynomial $P$ to be $N$-independent complex constants, and $P(X_1,\dots,X_k)$ to be self-adjoint if its arguments $X_1,\dots,X_k$ are.
\end{itemize}
\end{assumption}

By construction, $P_{\beta,N}=P(W_1^{(N)},\dots,W_k^{(N)})$ is a random matrix with a complicated correlation structure. However, thanks to Voiculescu \cite{Voiculescu91}, we know that the empirical measure $\hat\mu_{P_{\beta,N}}$ of the eigenvalues $\Spec(P_{\beta,N})=(\lambda_i(P_{\beta,N}))_{1\le i\le N}\in \R^N$ of $P_{\beta,N}$, denoted
$$\hat\mu_{P_{\beta,N}}=\frac{1}{N}\sum_{i=1}^N\delta_{\lambda_i(P_{\beta,N})}\, ,$$
converges almost surely and in expectation towards a probability measure $\mu_P$ which is the distribution of the spectrum of the operator $P(S_1,\ldots,S_k)$, where $S_1,\dots,S_k$ are free semi-circular elements. The measure $\mu_P$ has a  connected  compact support because the $C^*$-algebra of the free semi-circular variables is projectionless \cite{VoiculescuPim}. Moreover, by \cite{HaagerupTobjornsen2005, HST2006}, we know that the support of the spectrum $\Spec(P_{\beta,N})$ of the self-adjoint matrix $P_{\beta,N}$ converges towards the support of $\mu_P$. As a consequence, the largest eigenvalue $\lambda_1(P)=\lambda_1(P_{\beta,N})$ (resp. the smallest eigenvalue $\lambda_N(P)=\lambda_N(P_{\beta,N})$) of $P(W_1^{(N)},\dots,W_k^{(N)})$ converges almost surely towards the right boundary $r_P$ (respectively the left boundary $\ell_P$) of the support of $\mu_P$.

\subsection{Main Results}
We are aiming to show the following:

\begin{theorem}\label{thm-main}
Let $P_{\beta,N}=P(W_1^{(N)},\ldots,W_k^{(N)})$ such that Assumption~\ref{assum-main} holds.
\begin{itemize}
\item[(i)] The spectrum $\Spec(P_{\beta,N})$ of $P_{\beta,N}=P(W_1^{(N)},\ldots,W_k^{(N)})$ satisfies a large deviation principle with speed $\beta N$ and a  good rate function $\mathcal{I}_{\Spec(P)}^\beta$. Namely, for every closed subset $F$ of $\R$,
$$\limsup_{N\to\infty}\frac{1}{\beta N}\ln \P(\Spec(P_{\beta,N})\cap F\neq \emptyset)\le-\inf_F\mathcal{I}_{\Spec(P)}^{\beta}\,,$$
and for every open set $O\subset\R$
$$\liminf_{N\to\infty}\frac{1}{\beta N}\ln \P(\Spec{(P_{\beta,N})}\cap O\neq \emptyset)\ge-\inf_O\mathcal{I}_{\Spec(P)}^{\beta}\,.$$
$\mathcal{I}_{\Spec{(P)}}^{\beta}$ increases above $r_P$, decreases below $\ell_P$, is such that $\mathcal{I}_{\Spec{(P)}}^{\beta}(x)=0$ on the support of $\mu_P$, including $r_P$ and $\ell_P$.  $\mathcal I_{\Spec{(P)}}^{\beta}$ is positive on $]-\infty,\ell_P[\cup ]r_P,+\infty[$.
\item[(ii)] The largest eigenvalue of $P(W_1^{(N)},\dots,W_k^{(N)})$ satisfies a large deviation principle with speed $\beta N$ and an increasing good rate function $\mathcal{I}_{\lambda_{\max}(P)}^{\beta}$. Moreover, $\mathcal{I}_{\lambda_{\rm{max}}(P)}^{\beta}(x)=\mathcal{I}_{\Spec(P)}^{\beta}(x)$  for $x\ge r_P$, whereas $\mathcal{I}_{\lambda_{\rm{max}}(P)}^{\beta}(x)=+\infty$ for $x<r_P$.
\item[(iii)]{The operator norm  of $P(W_1^{(N)},\ldots,W_k^{(N)})$ satisfies a large deviation principle with speed $\beta N$ and a good rate function $\mathcal{I}_{\|P\|_{\rm{op}}}^{\beta}$. Moreover, $\mathcal{I}_{\|P\|_{\rm{op}}}^{\beta}(x)$ is infinite if $x<\max\{|r_P|,|\ell_P|\}$ and otherwise equals 
$\min\{\mathcal{I}_{\Spec(P)}^{\beta}(x),\mathcal{I}_{\Spec(P)}^{\beta}(-x) \}$.}
\end{itemize}
\end{theorem}
Note that the rate functions are not convex in general. When $k=1$ and $P(x)=x$, the rate function is known \cite{BenarousDG01} : $\mathcal I_{\lambda_{\rm{max}}(X)}^\beta(x)=\frac{1}{2}\int_2^x\sqrt{y^2-4} dy$ on $[2,+\infty)$. It is convex but $\mathcal I_{{\lambda_{\rm{max}}(X^3)}}^\beta(x)=\mathcal I_{\lambda_{\rm{max}}(X)}^\beta(x^{1/3})$ is not for instance.
The proof of Theorem~\ref{thm-main} for the case $\beta=1$ is given in Section~\ref{Sec:proofthm}, and the modifications for the case $\beta=2$ are outlined in Section~\ref{sect-betas}. We will see that $\mathcal{I}_{\Spec(P)}^{\beta}(z)=\mathcal{I}_{L_z}$ for $z\in\R$ with $\mathcal{I}_{L_z}^{\beta}$ defined in Theorem \ref{thm-Lbounds}, based on a linearization $L_z^{(N)}$ of the polynomial $P$. Note that a given polynomial can be associated with many different linearizations, however, our result implies that $\mathcal{I}_{L_z}^{\beta}$ only depends on $P$ and $z$. Although the rate function looks complicated at first glance, $\mathcal{I}_{L_z}^{\beta}$ is given by a $\theta\to\infty$ limit and involves only finite dimensional objects. We, therefore, expect the resulting rate function to be accessible numerically. 

\medskip
Theorem~\ref{thm-main} generalizes to a joint large deviation principle for the spectrum of several polynomials. 
\begin{corollary}\label{cor-multiP}
Let $r$ be an integer and consider $r$ polynomials $P_1,\ldots, P_r$ satisfying Assumption \ref{assum-main}. Then, the distribution of  

\noindent
$(\Spec(P_1(W_1^{(N)}, \ldots, W_k^{(N)})),\ldots, \Spec(P_r(W_1^{(N)}, \ldots, W_k^{(N)})))$ satisfies a large deviation principle with speed $\beta N$ and a good rate function $\mathcal{I}_{\Spec(P_i), 1\le i\le r}^{\beta}$.
\end{corollary}
We prove this generalization in Section \ref{sec-multiP}. Note that this corollary implies that for every positive numbers $x_1,\ldots,x_r$ which are continuity points of $\mathcal{I}_{\Spec(P_i), 1\le i\le r}^{\beta}$,
$$\limsup_{N\to\infty}\frac{1}{\beta N}\log \P( \|P_i(W_1^{(N)}, \ldots, W_k^{(N)})\|_{\mathrm{op}}\ge x_i, 1\le i\le r)
=-\inf_{\prod [x_i,+\infty)}\mathcal{I}_{\Spec(P_i), 1\le i\le r}^{\beta},$$ 
with inequalities as in \eqref{wldub} and \eqref{wldlb} at general points.
However, this does not allow a priori to estimate the probability that the norm of each polynomial is close to $x_i$. In fact, the argument of monotonicity of the rate function that we use when $r=1$ does not hold a priori when $r>1$. We shall investigate this point in a forthcoming work.

\subsection{General Notation}
We conclude the introduction with a brief overview of the notation and conventions used throughout the paper. The identity matrix of size $N\times N$ is written as $\Id_N$. For a matrix~A, the transpose and Hermitian conjugate are denoted by $A^T$ and $A^*$, respectively. Our default norm is the operator norm $\|\cdot\|_{\rm{op}}$ for matrices and the Euclidean norm $\|\cdot\|_2$ for vectors. The eigenvalues and eigenvectors of a matrix $A\in\R^{N\times N}$ or $A\in\C^{N\times N}$ are denoted as~$\lambda_j(A)$ and~$v_j(A)$ for $j\in\{1,\dots,N\}$, respectively, and eigenvectors are always assumed normalized to~$\|v_j(A)\|_2=1$. 

\medskip
Whenever the matrix $X$ is self-adjoint, we assume that its real eigenvalues are ordered as $\lambda_1(X)\geq\dots\geq \lambda_N(X)$. In particular, $\lambda_1$ always denotes the largest eigenvalue. We further introduce the empirical measure of the eigenvalues of $X$ as
\begin{displaymath}
\hat\mu_X=\frac{1}{N}\sum_{i=1}^N\delta_{\lambda_i(X)}\, ,
\end{displaymath}
and always include the underlying matrix as a subscript. Moreover, $\mathrm{Sym}_N(\R)$ denotes the symmetric matrices in $\R^{N\times N}$. We further add a superscript $\mathrm{Sym}_N^+(\R)\subset\R^{N\times N}$ to denote the symmetric positive semi-definite matrices and further occasionally use the set
\begin{displaymath}
\mathrm{Sym}_N^{+,1}(\R):=\{X\in\mathrm{Sym}_N^+(\R): \Tr(X)=1\}
\end{displaymath}
for matrices in $\mathrm{Sym}_N^+(\R)$ with the trace rescaled to one. Lastly, $\mathrm{Sym}_N(\C)$ denotes the Hermitian matrices in $\C^{N\times N}$, and the sets $\mathrm{Sym}_N^+(\C)$ and $\mathrm{Sym}_N^{+,1}(\C)$ are defined in analogy to the real case. For statements that include both cases, we introduce $\K_\beta$ and use it to denote $\R$ when $\beta=1$ and $\C$ when $\beta=2$.

\section{Preliminaries}

\subsection{Tilting by Spherical Integrals}
For $\beta\in\{1,2\}$,  $\theta\in\R$ and a deterministic self-adjoint matrix $H\in\mathrm{Sym}_N(\K_\beta)$, the spherical integral is defined as
\begin{equation}\label{eq-sphint}
I_{\beta,N}(\theta,H)=\int_{\mathds{S}^{\beta N-1}}\exp(\tfrac{\beta}{2} \theta N\langle u, H u\rangle)\dx u=\E_u[\exp(\tfrac{\beta}{2}\theta N\langle u, H u\rangle)]\, ,
\end{equation}
where $\E_u$ denotes expectation with respect to the uniform law on the $N$-dimensional unit sphere $\mathds{S}^{\beta N -1}$. The spherical integral~\eqref{eq-sphint} is a useful tool, because its behavior in the limit $N\rightarrow\infty$ is well-studied. More precisely, for a suitable sequence $(H_N)_N$ of self-adjoint matrices of increasing size, we have the following convergence result due to~\cite{GuionnetMaida2005,Maida2007}.

\begin{lemma}\label{lem-convsphint}
Let $(H_N)_N$ be a sequence of $N\times N$ self-adjoint deterministic matrices such that $\sup_{N}\|H_N\|_{\rm{op}}<\infty$, the empirical eigenvalue measures $(\hat \mu_{H_N})_N$ converge weakly towards a probability measure $\nu$, and the largest eigenvalues $(\lambda_1(H_N))_N$ converge to some number $y\in\R$. Then, for every $\theta\ge 0$,
\begin{equation}\label{eq-convsphint}
\lim_{N\rightarrow\infty}\frac{1}{\beta N}\ln[I_{\beta,N}(\theta,H_N)]=J_{\nu}(y,\theta)\, ,
\end{equation}
with an explicit function $J_\nu(y,\theta)$ that only depends on $\nu$, its Stieltjes transform, and the arguments $y$ and $\theta$. Moreover, the convergence in~\eqref{eq-convsphint} is uniform on small neighborhoods of $\hat\mu_{H_N}$ (for the weak topology) and $\lambda_1(H_N)$. The same result holds if $\theta$ is negative and the largest eigenvalue is replaced by the smallest eigenvalue.
\end{lemma}

Note that prefactor $(\beta N)^{-1}$ in~\eqref{eq-convsphint} matches the speed that we are aiming for in Theorem~\ref{thm-main}. The function $J_\nu$ can be given explicitly, and we recall its definition from~\cite{GuionnetMaida2005,Maida2007} for further use. First, recall that the Stieltjes transform $m_\nu$ of a probability measure $\nu$ is given by
$$
m_\nu(z)=\int\frac{\nu(\dx x)}{x-z}\,.
$$
If $r_\nu$ denotes the rightmost point in the support of $\nu$, $m_\nu$ is strictly increasing on $(r_\nu,+\infty)$ with values in $(\lim_{y\downarrow r_\nu} m_\nu(y),0)$ and we denote 
by $(-m_\nu)^{-1}$ the functional inverse of its negative from $(0,-\lim_{y\downarrow r_\nu} m_\nu(y))$ into $(r_\nu,\infty)$. For $\theta\geq0$, we have
\begin{align}
&J_\nu(x,\theta)\label{eq-defJ}\\
&:=\begin{cases}\frac{1}{2}\left(
\theta (-m_\nu)^{-1}(\theta)-(1+\ln(\theta))-\int_\R\ln\big((-m_\nu)^{-1}(\theta)-y\big)\dx\nu(y)\right), &\text{if }\theta\le-m_\nu({x}),\\
\frac{1}{2}\left(\theta x-(1+\ln(\theta))-\int_\R \ln|x-y|\dx\nu(y)\right),&\text{if }\theta\ge -m_\nu({x}).\\
\end{cases}\NN
\end{align}
To use the spherical integral techniques to prove large deviations, a key point is to compute the average of the spherical integral with respect to the sequence $H_N$ of random matrices under study. This is a difficult question if $H_N=P_N$ except if $P_N$ depends linearly on the matrices $(W_1^{(N)},\ldots,W_k^{(N)})$. This is why we will first consider a suitable linearization of the polynomial.

\subsection{Linearization}
Our key tool to treating polynomials in GOE or GUE matrices is the following ``linearization trick'' introduced in~\cite{HST2006, HaagerupTobjornsen2005}: Given a self-adjoint polynomial $P$ in GOE or GUE matrices $W_1^{(N)},\dots, W_k^{(N)}$, one can find a linearization $L_z^{(N)}$ such that $L_z^{(N)}$ is non-invertible iff~$z$ is in the spectrum of $P(W_1^{(N)},\dots, W_k^{(N)})$. Similar to the notation for the Gaussian Wigner matrices, we drop the explicit $N$-dependence from $L_z^{(N)}$ in the following derivation. First, write the polynomial $P(X_1,\dots,X_k)$ with (deterministic) arguments $X_1,\dots,X_k$ as a matrix product $P = M_1 \dots M_{p}$ where the $M_j$ are matrices of compatible dimensions and whose entries are polynomials of degree $1$ or less. We then form for every $z \in\C$ the linearization $L_z$ as 
\begin{equation}\label{eq-linearization}
L_z := \begin{pmatrix} z &  -M_1 & 0 & \dots & 0 \\
0& \Id & - M_2 & \dots & 0 \\
\vdots & & &\ddots & \\
 -M_{p} & 0& \hdots & & \Id \end{pmatrix}\, .
\end{equation}
The blocks $M_i$ may have diﬀerent sizes $l_i\times r_i$ and $\ell=\sum_{1\le i\le p} l_i=\sum_{1\le i\le p} r_i$ below stands for the total size of the block matrix.
Note that such a linearization is of the form
\begin{equation}\label{eq-defLz}
L_z=\sum_{j=1}^kA_j\otimes X_j+A_0\otimes\Id_N
\end{equation}
with suitable matrices $A_0,\dots,A_k$ in $\R^{\ell\times\ell}$ ($\beta=1$) or in $\C^{\ell\times\ell}$ ($\beta=2$) that, a priori, do not need to be self-adjoint. We then have (see, e.g., Proposition 2.3 in \cite{HST2006}) that $L_z$ is invertible if and only if $(z - M_1\dots M_p) $ is and if it is the case then $L_{z}^{ -1} = B(z) + C $, where

\begin{align}
B(z) &= \begin{pmatrix} \Id \\ M_2 \dots M_p \\ \vdots \\ M_p \end{pmatrix} ( z - M_1 \dots M_p)^{ -1} (\Id \quad M_1 \quad M_1M_2 \quad \dots \quad M_1 M_2\dots  M_{p-1} )\nonumber\\
&=: V_1 (z - P)^{ -1} V_2^T\label{poiu}
\end{align}
and
\begin{displaymath}
C = \begin{pmatrix} 0 & 0 & \hdots & & & &\\
0 & \Id & M_2 & M_2 M_3 & \hdots& & \\
0 & 0 & \Id & M_3 & \hdots & & \\
\vdots &  & &  & & & \\
0 & \hdots & & & & & \Id \end{pmatrix}\, .
\end{displaymath}

In particular, we notice that $(z -P)^{-1}$ is in position $(1,1)$ of $(L_z)^{ -1}$. Hence, once we substitute the matrices $W_1^{(N)},\ldots,W_k^{(N)}$ for $X_1, \dots, X_k$, and denote the resulting matrix as $L_z^{(N)}$, 
we have $$\| (L_z^{(N)})^{-1}\|_{\rm{op}}=\sup_{\|u\|_2=\|v\|_2=1} |\langle v, (L_z^{(N)})^{-1} u\rangle|
\ge \sup_{\|u\|_2=\|v\|_2=1} |\langle v, ((L_z^{(N)})^{-1})_{1,1} u\rangle|= \| (z - P)^{-1}\|_{\rm{op}},$$ where we restricted the supremum to vectors supported on the first coordinate. This implies  that if $\lambda_{N\ell}(L_zL_z^*)$ denotes the smallest eigenvalue of $L_z^{(N)} (L_z^{(N)})^*$, $$\| (L_z^{(N)})^{-1}\|^{-1}_{\rm{op}}=\sqrt{\lambda_{N\ell} (L_z L_z^*)} \leq \| (z - P)^{-1}\|_{\rm{op}}^{-1}= \dist(z, \Spec(P_N))\,,$$
and consequently,
for $\varepsilon\in (0,1)$,
$$\P[ \text{dist}(z,\Spec(P_N)) \leq \epsilon ] \leq \P [ \lambda_{N\ell}(L_zL_z^*) \leq  \epsilon^2 ]\,.$$
Furthermore, for every $M >0$, on the event $\Omega_M=\{\max_j \|W_j^{(N)}\|_{\rm{op}}\le M\}$, we find a finite constant $c_1(M)$ such that we have $\|V_1\|_{\rm{op}}, \|V_2\|_{\rm{op}}, \|C\|_{\rm{op}} \leq c_1(M)$. On the other hand, by concentration  of the eigenvalues of Gaussian matrices, see e.g. \cite[Thm.~6.6.2]{AliceStFlour}, and therefore of their norm, we know that there exists $M_0>0$ such that for $M\ge M_0$
\begin{equation}\label{exptlambda}
\P(\|W_i^{(N)}\|_{\rm{op}}\ge M)\le e^{-\frac{\beta}{8} N M^2}\,.
\end{equation}
Moreover, on $\Omega_M$, we have  by \eqref{poiu}
$$\| (L_z^{(N)})^{-1}\|_{\rm{op}} \leq c_1(M) (\| (z - P)^{-1}\|_{\rm{op}} + 1),$$ i.e. we find on $\Omega_M$ that if
$c_1(M) \sqrt{ \lambda_{N\ell} (L_zL_z^*)}<1$,

\begin{displaymath}
\dist(z, \Spec(P_N)) \leq  \Big( \frac{1}{c_1(M) \sqrt{ \lambda_{N\ell} (L_zL_z^*)}} -1 \Big)^{ -1}\, . 
\end{displaymath}
In short, for $M >M_0$ and $\delta\in [0,c_1(M)^{-2}[$,
we get
$$\P [ \lambda_{N\ell}(L_z L_z^*) \leq \delta ] \le  e^{- \frac{\beta}{8} N M^2} +  \P\Bigl[ \text{dist}(z,\Spec(P_N)) \le \frac{c_1(M) \sqrt{\delta}}{1-c_1(M)\sqrt{\delta}}\Bigr]\, .$$
Thus, we  can find two positive functions $\delta_0(\varepsilon)=(c_1(M)^{-1} \varepsilon(1+\varepsilon)^{-1})^{2},\delta(\varepsilon)=\varepsilon^2$ such that
\begin{align*}
\P [ \lambda_{N\ell}(L_zL_z^*) \leq \delta_0( \epsilon) ] - e^{- \frac{\beta}{8} N M^2} &\leq  \P[ \text{dist}(z,\Spec(P_N))\leq \epsilon ]\\
& \leq \P [ \lambda_{N\ell}(L_zL_z^*) \leq \delta( \epsilon) ]\, .
\end{align*}

Therefore, to understand the asymptotics of $\P[ \dist(z, \Spec(P_N)) \leq \epsilon ]$, we only need to consider $\P [ \lambda_{N\ell}(L_zL_z^*) \leq \delta( \epsilon) ]$. We have that
\begin{align}
\lim_{\epsilon \to 0} \limsup_N \frac{1}{N} \ln \P [ \dist(z, \Spec(P_N)) \leq \epsilon ]  &= \lim_{\epsilon \to 0} \limsup_N \frac{1}{N} \ln \P [ \lambda_{N\ell}(L_zL_z^*) \leq \epsilon ]\label{eq-linUB}\\
\lim_{\epsilon \to 0} \liminf_N \frac{1}{N} \ln \P [ \dist(z, \Spec(P_N)) \leq \epsilon ]  &= \lim_{\epsilon \to 0} \liminf_N \frac{1}{N} \ln \P [ \lambda_{N\ell}(L_zL_z^*) \leq \epsilon ]\label{eq-linLB}.
\end{align}

\medskip
In view of~\eqref{eq-linUB} and~\eqref{eq-linLB}, a large deviation principle for the spectrum and the  largest eigenvalue of $P_N$ may be deduced from estimating the probability that $\lambda_{N\ell}(L_z L_z^*)$ is small for every $z\in \R$. This is detailed in Section \ref{Sec:proofthm}.
Therefore, we aim to prove the following:
\begin{theorem}\label{thm-Lbounds}
Under Assumption \ref{assum-main}, we have for every $z\in\R$,
$$
\lim_{\epsilon\rightarrow0}\liminf_{N\rightarrow\infty}\frac{1}{\beta N}\ln\P[\lambda_{N\ell}(L_zL_z^*)\leq\epsilon]
=\lim_{\epsilon\rightarrow0}\limsup_{N\rightarrow\infty}\frac{1}{ \beta N}\ln\P[\lambda_{N\ell}(L_zL_z^*)\leq\epsilon]=-\mathcal{I}_{L_z}^{\beta},
$$
with
$$-\mathcal{I}_{L_z}^{\beta}:=
\lim_{\theta\to+\infty}\sup_{\Psi\in \mathrm{Sym}^{+,1}_\ell(\K_\beta)}[F^{(\beta)}(\theta,\Psi)-\widetilde{\mathcal{I}}(\Psi)-\ell J_{\nu_z}(\theta)]\, .$$
where $F^{(\beta)}$ is defined in~\eqref{eq-defF} below for $\beta=1$ and in~\eqref{eq-defF-GUE} for $\beta=2$ and depends on the matrices $A_0,\dots,A_k$ in~\eqref{eq-defLz}. Moreover, $\widetilde{\mathcal{I}}$ is the rate function in~\cite[Lem.~3.1]{LDP_Kronecker} (cf. Lemma~\ref{lem-WishartLDP} below), and $J_{\nu_z}(\theta)$ is the limit of the spherical integral in~\eqref{eq-convsphint} evaluated at {$-\theta$}, the appropriate limiting measure $\nu_z$ of the empirical distribution of the eigenvalues of $L_z^{(N)}(L_z^{(N)})^*$ and the smallest eigenvalue $0$, namely $J_{\nu_z}(0,-\theta)$.
\end{theorem}

The proof of Theorem~\ref{thm-Lbounds} uses the technique of tilting by spherical integrals and is split into an upper and a lower bound. We show these bounds for the case $\beta=1$ in Sections~\ref{sect-UB} and~\ref{sect-LB}, respectively, and discuss the necessary modifications for $\beta=2$ in Section~\ref{sect-betas}. Before diving deeper into the argument, however, we note the following properties of $L_z^{(N)}(L_z^{(N)})^*$.

\begin{lemma}\label{lem-prelimbounds}
Let $L_z^{(N)}$ be defined as in~\eqref{eq-linearization} above  with $(X_1,\ldots,X_k)=(W_1^{(N)},\ldots,W_k^{(N)})$ and denote by $\hat \mu_{L_z^{(N)}(L_z^{(N)})^*}$ the empirical spectral measure of $L_z^{(N)}(L_z^{(N)})^*$. It holds that
\begin{itemize}
\item[(i)] There exists a deterministic probability measure $\nu_z$ such that $\hat \mu_{L_z^{(N)}(L_z^{(N)})^*}\rightarrow\nu_z$ weakly in probability.
\item[(ii)] For every $\delta>0$,
\begin{displaymath}
\lim_{N\rightarrow\infty}\P(|\lambda_{N\ell}( L_z^{(N)}(L_z^{(N)})^*)-\ell_z |\geq\delta)=0, \lim_{N\rightarrow\infty}\P(|\lambda_{1}( L_z^{(N)}(L_z^{(N)})^*)-r_z |\geq\delta)=0\, ,
\end{displaymath}
where $\ell_z=\min\supp\nu_z$ and $r_z=\max\supp\nu_z$.
\item[(iii)] The sequence $(\|L_z^{(N)}(L_z^{(N)})^*\|_{\rm{op}})_N$ is exponentially tight. More precisely, there exists a constant $c>0$ such that for every $M>0$ it holds that
\begin{displaymath}
\limsup_{N\to\infty}\frac{1}{N}\ln\P(\|L_z^{(N)}(L_z^{(N)})^*\|_{\rm{op}}>M)\leq c-\frac{\ell\sqrt{M}}{4}\,,
\end{displaymath}
for a finite constant $c$ depending only on the linearization (as well as $\ell$).
\end{itemize}
\end{lemma}
In the case where the matrices $A_0,\dots,A_k$ are self-adjoint, $L_z^{(N)}$ matches the definition of the Gaussian Kronecker matrix in~\cite{LDP_Kronecker}, and Lemma~\ref{lem-prelimbounds} is immediate from the discussion of the model in~\cite[Sect.~1.1]{LDP_Kronecker}, as well as Lemmas~1.2 and~4.5 there. However, we do not insist on a self-adjoint linearization here.

\medskip
We further need the following concentration bound.
\begin{lemma}\label{lem-conc} With the same notations as in Lemma \ref{lem-prelimbounds}, for all $\epsilon>0$, it holds that
\begin{displaymath}
\limsup_{N\to\infty}\frac{1}{N}\ln\P(d_{sq}(\hat \mu_{L_z^{(N)}(L_z^{(N)})^*},\nu_z)\geq \epsilon)=-\infty,
\end{displaymath}
where $d_{sq}$ is defined as
\begin{displaymath}
d_{sq}(\nu_1,\nu_2):=\sup_{f\in\mathds{L}}\int_{\R^+} f(x)(\nu_1(\dx x)-\nu_2(\dx x))
\end{displaymath}
for two probability measures $\nu_1,\nu_2$ supported on $\R^+$ and test functions in
\begin{equation}\label{eq-defdsq}
\mathds{L}:=\{f\text{ Lipschitz on }\R:\|f\|_\infty\leq 1, g:x\mapsto f(x^2) \text{ Lipschitz with }\mathrm{Lip}(g)\leq1\}
\end{equation}
where $\mathrm{Lip}(g)$ denotes the Lipschitz constant associated to a Lipschitz function $g$.

In particular, there exists a positive sequence $(\kappa_n)_{n \geq 0}$ going to $0$ at $\infty$ such that 
\begin{displaymath}
\limsup_{N\to\infty}\frac{1}{N}\ln\P(d_{sq}(\hat \mu_{L_z^{(N)}(L_z^{(N)})^*},\nu_z)\geq \kappa_N)=-\infty\, .
\end{displaymath}
\end{lemma}

The proof of Lemmas~\ref{lem-prelimbounds} and~\ref{lem-conc} readily follows from the existing literature. For the convenience of the reader, we include a brief proof in Appendix~\ref{sect-app}. We now turn to the proof of Theorem~\ref{thm-Lbounds}.

\section{Proof of Theorem~\ref{thm-Lbounds}: Upper Bound}\label{sect-UB}
We restrict the following argument to $\beta=1$ and discuss the modifications for $\beta=2$ in Section~\ref{sect-betas}. With $\beta$ fixed, we drop the parameter to simplify notation. Choosing a positive sequence $( \kappa_n)_{n \geq 0} $ given by Lemma \ref{lem-conc}, start by defining the set 
\begin{align*}
\Omega_{\alpha,\kappa}:=\{H:d_{sq}(\mu_{HH^*},\nu_z)< \kappa_N \}\cap\{H:\ \|HH^*\|_{\rm{op}}\leq\alpha\}\, .
\end{align*}
Note that the probability that $L^{(N)}_z$ belongs to $\Omega_{\alpha,\kappa}$ is overwhelming by Lemmas~\ref{lem-prelimbounds}(iii) and~\ref{lem-conc}. Let us also abbreviate
\begin{equation}\label{eq-defOm}
\Omega_{\eps,\alpha,\kappa}:=\{H:\ \lambda_{N\ell}(HH^*)\leq\epsilon\}\cap\Omega_{\alpha,\kappa}
\end{equation}
for the following argument. Next, fix $\theta>0$. We apply the tilt and use Lemma \ref{lem-convsphint} to find
\begin{align}
\frac{1}{N}\ln\P(L_z^{(N)}\in \Omega_{\epsilon,\alpha,\kappa})	&= \frac{1}{N} \ln \E[ \mathds{1}_{\{L_z^{(N)}\in\Omega_{\epsilon,\alpha,\kappa}\}} \frac{I_{N\ell}(-\theta,L_z^{(N)}(L_z^{(N)})^*)}{I_{N\ell}(-\theta,L_z^{(N)}(L_z^{(N)})^*)}  ] \nonumber\\
&\leq \frac{1}{N} \ln \E_W[ \mathds{1}_{\{L_z^{(N)}\in \Omega_{\epsilon,\alpha,\kappa}\}}I_{N\ell}(-\theta,L_z^{(N)}(L_z^{(N)})^*)  ]  - \ell J_{\nu_z}(\theta) +o(1)\nonumber\\
&\leq \frac{1}{N} \ln \E_W[I_{N\ell}(-\theta,L_z^{(N)}(L_z^{(N)})^*)] - \ell J_{\nu_z}(\theta) +o(1)\, ,\label{eq-ub1}
\end{align}
which leaves us with computing the limit of the annealed spherical integral. Here, the $o(1)$ error goes to zero as $N$ goes to infinity and then $\epsilon$ goes to zero while $\theta$ remains fixed. Noting that
\begin{displaymath}
\E_W[I_{N\ell}(-\theta,L_z^{(N)}(L_z^{(N)})^*)]=\E_u\E_W[\exp(-\tfrac{\theta}{2} N\ell\|L_z^{(N)}u\|^2)]
\end{displaymath}
by Fubini's theorem, the first step is considering the inner integral. 

\medskip
We start by introducing the following notation. For $u\in\mathds{S}^{\ell N-1}$, write $u=(u_1,\dots,u_\ell)$ with blocks $u_1,\dots,u_\ell$ of length $N$. Similar to~\cite{LDP_Kronecker}, we define the matrix $\rho(u)$ as
\begin{displaymath}
\rho(u)_{ij}=\langle u_i,u_j\rangle, \quad i,j=1,\dots,\ell.
\end{displaymath}
Note that this matrix is positive semi-definite by definition and that $\Tr(\rho(u))=\|u\|^2=1$. It turns out that $\rho(u)$ is the key to computing $\E_W[\exp(-\frac{1}{2}\theta N\ell\|L_z^{(N)}u\|^2)]$ at fixed $u$.

\begin{lemma}\label{lem-annealedcalc}
Under Assumption~\ref{assum-main}, it holds that
\begin{displaymath}
\E_W[\exp(-\tfrac{\theta}{2}N\ell\|L_z^{(N)}u\|^2)]=\exp(f_N(\theta,\rho(u)))
\end{displaymath}
for a suitable continuous function $f_N$ given by
\begin{equation}\label{eq-deffN}
f_N(\theta,\Psi)=-\frac{N\ell\theta}{2}\,Q_\theta(\Psi)-\frac{N}{2}\ln(\det(B))-\frac{1}{2}\ln\det\big(\Id_{\ell^2}+\theta\ell\,B^{-1}\mathcal{R}_\Psi\big),
\end{equation}
where
\begin{equation}\label{eq-defB}
C_\Psi:=\sum_{j=1}^k A_j \Psi A_j^T,\qquad B:= \Id_{\ell} + \theta\ell\, C_\Psi ,
\end{equation}
and where, for $\Psi\in\mathrm{Sym}^{+}_\ell(\R)$, we consider the linear maps on $\R^{\ell\times\ell}$ given by
\begin{align*}
\mathcal{R}_\Psi(Z):=\sum_{j=1}^k A_j\Psi^{1/2}Z^TA_j\Psi^{1/2}\, ,\quad \mathcal{K}_\Psi(Z):=C_\Psi Z+\mathcal{R}_\Psi(Z)\, ,
\end{align*}
and set $(B^{-1}\mathcal{R}_\Psi)(Z):=B^{-1}\mathcal{R}_\Psi(Z)$. The map $\mathcal{K}_\Psi$ is self-adjoint and non-negative for $\langle\cdot,\cdot\rangle_F$, and
\begin{equation}\label{eq-defQ}
Q_\theta(\Psi):=\big\langle A_0\Psi^{1/2},\,(\Id_{\ell^2}+\theta\ell\,\mathcal{K}_\Psi)^{-1}(A_0\Psi^{1/2})\big\rangle_F\, ,
\end{equation}
where $\langle Y,Z\rangle_F=\Tr(Y^TZ)$ denotes the Frobenius inner product.
\end{lemma}

\begin{proof}
Fix $u\in\mathds{S}^{N\ell-1}$ and abbreviate $\rho(u)=:\Psi$ for the following computation. Recalling that the linearization is of the form $L_z^{(N)}=\sum_{j=1}^kA_j\otimes W_j^{(N)}+A_0\otimes \Id_N$ with GOE matrices $W_1^{(N)},\dots,W_k^{(N)}$, and deterministic $A_0,\dots,A_k$, we conclude that $L_z^{(N)}u$ is a Gaussian random vector. To compute its covariance, recall that for a GOE matrix $W$ and deterministic $x,y\in\R^N$, it holds that
\begin{displaymath}
\E\big[(Wx)(Wy)^T\big]=\frac1N\big(\langle x,y\rangle\,\Id_N+y\,x^T\big)\, .
\end{displaymath}
Write $v^{(j)}:=(A_j\otimes\Id_N)u$, with blocks $v^{(j)}_a=\sum_{b}(A_j)_{ab}u_b\in\R^N$, $a=1,\dots,\ell$. Using the independence of the $W_j$, the matrix $\Sigma := N \Cov(L_z^{(N)}u)$ is the $N\ell\times N\ell$ matrix whose $N\times N$ block $(a,a')$ equals
\begin{equation}\label{eq-cov-correct}
\Sigma_{aa'}=(C_\Psi)_{aa'}\,\Id_N+\sum_{j=1}^k v^{(j)}_{a'}\big(v^{(j)}_{a}\big)^T\, .
\end{equation}
Note that the second term is a partial transpose of $\sum_j(A_j\otimes\Id_N)uu^T(A_j^T\otimes\Id_N)$, whose block $(a,a')$ is $v^{(j)}_a(v^{(j)}_{a'})^T$. 
This allows using the following Gaussian integration formula: For $\theta \geq 0$ and a real Gaussian vector $Y$ with expectation $\xi_Y$ and covariance matrix $\Sigma_Y$, we have
\begin{displaymath}
\E[ \exp (- \theta \|Y\|^2) ] = \frac{\exp( - \theta 
 \langle \xi_Y, (\Id_{N \ell} + 2\theta \Sigma_Y)^{ -1}  \xi_Y \rangle ) }{\sqrt{\det(\Id_{N \ell} + 2\theta \Sigma_Y)}}\, .
\end{displaymath}

Therefore, if we denote $\xi=\E_W[L_z^{(N)}u]$ and $\Sigma = N \Cov(L_z^{(N)}u) $, we have that 
\begin{align}
\E_W[ \exp ( - \tfrac{\theta}{2} N\ell \|L_z^{(N)}u\|^2) ]
&= \frac{\exp( - \tfrac{\theta}{2} N\ell  \langle \xi,  (\Id_{N \ell} + \theta\ell \Sigma)^{ -1} \xi \rangle ) }{\sqrt{\det(\Id_{N \ell} + \theta\ell \Sigma)}}\, .\label{eq-expectationLu}
\end{align}

Next, we express~\eqref{eq-expectationLu} as a function of $\Psi$. Assume first that $u_1,\dots,u_\ell$ are linearly independent, i.e.\ that $\Psi$ is invertible, and let $\mathcal{V}_u:=\mathrm{span}\{e_a\otimes u_b:\,1\leq a,b\leq\ell\}\subset\R^{N\ell}$. Define
\begin{displaymath}
\mathcal{U}:\R^{\ell\times\ell}\to\mathcal{V}_u,\qquad \mathcal{U}(Z):=\big((Z\Psi^{-1/2})\otimes\Id_N\big)u ,
\end{displaymath}
i.e., the $a$-th block of $\mathcal{U}(Z)$ is equal to $\sum_b(Z\Psi^{-1/2})_{ab}u_b$. Since $\langle(X\otimes\Id_N)u,(Y\otimes\Id_N)u\rangle=\Tr(X\Psi Y^T)$, $\mathcal{U}$ is an isometry from $(\R^{\ell\times\ell},\langle\cdot,\cdot\rangle_F)$ onto $\mathcal{V}_u$. We claim that:
\begin{enumerate}
\item[(a)] $\xi=\mathcal{U}(A_0\Psi^{1/2})$;
\item[(b)] $(C_\Psi\otimes\Id_N)\,\mathcal{U}(Z)=\mathcal{U}(C_\Psi Z)$, so that $C_\Psi\otimes\Id_N$ leaves $\mathcal{V}_u$ and $\mathcal{V}_u^\perp$ invariant;
\item[(c)] The second term $\Sigma':=\Sigma-C_\Psi\otimes\Id_N$ in~\eqref{eq-cov-correct} satisfies $\Sigma'\mathcal{U}(Z)=\mathcal{U}(\mathcal{R}_\Psi(Z))$ and $\Sigma'=0$ on $\mathcal{V}_u^\perp$.
\end{enumerate}
Points (a) and (b) are immediate from an explicit computation. For (c), let $x=\mathcal{U}(Z)$, so that $x_{a'}=\sum_cX_{a'c}u_c$ with $X=Z\Psi^{-1/2}$. Then $\langle v^{(j)}_a,x_{a'}\rangle=(A_j\Psi X^T)_{aa'}$, hence
\begin{displaymath}
(\Sigma'x)_a=\sum_{j=1}^k\sum_{a'=1}^\ell v^{(j)}_{a'}\,(A_j\Psi X^T)_{aa'}=\sum_{j=1}^k\sum_{d=1}^\ell(A_j\Psi X^TA_j)_{ad}\,u_d\, ,
\end{displaymath}
so that $\Sigma'x=\mathcal{U}(Z')$ with $Z'=\sum_jA_j\Psi X^TA_j\Psi^{1/2}=\mathcal{R}_\Psi(Z)$. If $x\perp\mathcal{V}_u$, every block $x_{a'}$ is orthogonal to $u_1,\dots,u_\ell$, hence to all $v^{(j)}_a$, and $\Sigma'x=0$.

\medskip
Consequently, $\mathcal{V}_u$ and $\mathcal{V}_u^\perp$ are invariant under $\Sigma$, we have $\Sigma=C_\Psi\otimes \Id_N$ on $\mathcal{V}_u^\perp$, and $\Sigma$ is conjugated by $\mathcal{U}$ to $\mathcal{K}_\Psi$ on $\mathcal{V}_u$. In particular $\mathcal{K}_\Psi$ is self-adjoint and non-negative, and, by~(a), we can now identify that
\begin{displaymath}
\langle \xi,  (\Id_{N \ell} +  \theta\ell \Sigma)^{ -1} \xi \rangle=Q_\theta(\Psi)\, .
\end{displaymath}
For the determinant in the denominator of~\eqref{eq-expectationLu}, we note that $\det(\Id_{N\ell}+\theta\ell\, C_\Psi\otimes\Id_N)=\det(B)^N$, and the restriction of $\Id+\theta\ell\, C_\Psi\otimes\Id_N$ to $\mathcal{V}_u$ is conjugated to $Z\mapsto BZ$, whose determinant is $\det(B)^\ell$. Since on $\mathcal{S}$, $\Id+\theta\ell\mathcal{K}_\Psi=B\,(\Id+\theta\ell B^{-1}\mathcal{R}_\Psi)$, we get
\begin{align*}
\det( \Id_{N \ell} +  \theta\ell \Sigma) &= \frac{\det(B)^{N }}{\det(B)^\ell}\cdot\det(B)^\ell\det\big(\Id_{\ell^2}+\theta\ell B^{-1}\mathcal{R}_\Psi\big)\\
&=\det(B)^{N}\det\big(\Id_{\ell^2}+\theta\ell B^{-1}\mathcal{R}_\Psi\big)\, .
\end{align*}
Plugging these two identities into~\eqref{eq-expectationLu} gives $\E_W[\exp(-\frac{\theta}{2} N\ell\|L_z^{(N)}u\|^2)]=\exp(f_N(\theta,\Psi))$ with $f_N$ as in~\eqref{eq-deffN}.

\medskip
It remains to check continuity. We note that $\Psi\mapsto\Psi^{1/2}$ is continuous on $\mathrm{Sym}_\ell^+(\R)$, $\Id+\theta\ell\mathcal{K}_\Psi\geq\Id$, and 
\begin{displaymath}
\det(\Id_{\ell^2}+\theta\ell B^{-1}\mathcal{R}_\Psi)=\det(B)^{-\ell}\det(\Id_{\ell^2}+\theta\ell\mathcal{K}_\Psi)>0\, ,
\end{displaymath}
where $\mathcal{K}_\Psi$ is non-negative for all $\Psi\geq0$ by continuity. Hence $f_N(\theta,\cdot)$ is indeed continuous on $\mathrm{Sym}_\ell^+(\R)$. Since both sides of the identity are continuous in $u$ and the set of $u$ with~$\rho(u)$ invertible is dense in $\mathds{S}^{N\ell-1}$, the identity holds for every $u$.
\end{proof}

Consider now a random vector $u$ following the uniform law on $\mathds{S}^{\ell N-1}$. Using that
\begin{displaymath}
u\overset{law}{=}\frac{g}{\|g\|}\, ,
\end{displaymath}
where $g$ is a vector of $\ell N$ independent standard Gaussian random variables, we find that the law of $\rho(u)$ is given by the law of an $\ell\times\ell$ Wishart matrix rescaled by its trace. We note the following large deviation principle for $\rho(u)$.

\begin{lemma}[{\cite[Lem.~3.1]{LDP_Kronecker}}]\label{lem-WishartLDP} For $\beta\in \{1,2\}$
let $u$ be a random vector following the uniform law on the sphere in $(\K_{\beta})^{\ell N}$. Then the (upper triangular) entries of $\rho(u)$ satisfy a large deviation principle in $\mathrm{Sym}_\ell^{+,1}(\K_{\beta})$ with speed $\beta N$ and good rate function
\begin{displaymath}
\widetilde{\rate}(X)=-\frac{1}{2}(\ln(\det(X))+\ell\ln(\ell))\, .
\end{displaymath}
\end{lemma}

Combining the above with Varadhan's lemma now gives the desired convergence result for the annealed spherical integral.

\begin{lemma}
Under  Assumption \ref{assum-main}, we have that for every $\theta\ge 0$
\begin{displaymath}
\lim_{N\rightarrow\infty}{\frac{1}{ N}}\ln\E[I_{ N\ell}(-\theta,L_z^{(N)}(L_z^{(N)})^*)]=\sup_\Psi\{F(\theta,\Psi)-\widetilde{\mathcal{I}}(\Psi)\}\, ,
\end{displaymath}
where $\widetilde{\mathcal{I}}$ was defined in Lemma~\ref{lem-WishartLDP} and
\begin{align}
F(\theta, \Psi)&= -\frac{\theta\ell}{2}\,Q_\theta(\Psi)- \frac{1}{2}\ln(\det(\Id_\ell+\theta\ell\, C_\Psi)) \label{eq-defF}
\end{align}
with $C_\Psi$ and $Q_\theta$ defined in~\eqref{eq-defB} and~\eqref{eq-defQ}, respectively.
\end{lemma}

\begin{proof}
From Lemma~\ref{lem-annealedcalc}, we have
\begin{align}
\E_W\E_u[\exp(-\tfrac{\theta}{2} N\ell\|L_z^{(N)}u\|^2)]&=\E_u\E_W[\exp(-\tfrac{\theta}{2} N\ell\|L_z^{(N)}u\|^2)]\nonumber\\
&=\E_u[\exp( f_N(\theta,\rho(u)))]\, ,\nonumber
\end{align}
where the function $f_N$ of the profile $\rho(u)=\Psi$ is of the form
\begin{displaymath}
f_N(\theta,\Psi)= NF(\theta,\Psi)+\widetilde{f}_N(\theta,\Psi).
\end{displaymath}
with $F$ in~\eqref{eq-defF}, which is a continuous  function of $\Psi$, and 
$$\widetilde{f}_N(\theta,\Psi)=-\frac12\ln\det(\Id_{\ell^2}+\theta\ell B^{-1}\mathcal{R}_\Psi)\,.$$
As $\Psi\leq\Id_\ell$, we have $\|C_\Psi\|\leq\sum_j\|A_j\|^2$ and $\|\mathcal{R}_\Psi\|\leq\sum_j\|A_j\|^2$. Writing $\widetilde{f}_N=\frac{\ell}{2}\ln\det B-\frac12\ln\det(\Id_{\ell^2}+\theta\ell\mathcal{K}_\Psi)$, we thus obtain the bound
\begin{displaymath}
\sup_\Psi|\widetilde{f}_N(\theta,\Psi)|\leq \ell^2\ln\Big(1+2\theta\ell\sum_{j=1}^k\|A_j\|^2\Big),
\end{displaymath}
which does not depend on $N$. In particular, $\lim_{N\rightarrow\infty}\frac{1}{N}\sup_\Psi|\widetilde{f}_N(\theta,\Psi)|=0$. Moreover, $\rho(u)$ satisfies a large deviation principle (Lemma~\ref{lem-WishartLDP}). Recalling that $\xi=\E_W[L_z^{(N)}u]$ and $\Sigma = N \Cov(L_z^{(N)}u) $, we note that $\|\xi\|\leq \|A_0\|$ and $\|\Sigma\|\leq2\sum_{j=1}^k\|A_j\|^2$. Since further $\|u\|$=1 by definition and $\Sigma$ is positive semi-definite, we readily obtain the necessary moment condition to apply Varadhan's lemma. Hence,
\begin{displaymath}
\lim_{N\rightarrow\infty}{\frac{1}{ N}}\ln\E[I_{N\ell}(-\theta,L_z^{(N)}(L_z^{(N)})^*)] =\sup_\Psi\{F(\theta,\Psi)-\widetilde{\mathcal{I}}(\Psi)\}\, ,
\end{displaymath}
where $\widetilde{\mathcal{I}}$ is the rate function for the profile from Lemma~\ref{lem-WishartLDP}.
\end{proof}
Recall that  the probability that $L^{(N)}_z$ belongs to $\Omega_{\alpha,\kappa}$ is overwhelming  by Lemmas~\ref{lem-prelimbounds}(iii) and~\ref{lem-conc}. We thus conclude the following upper bound from \eqref{eq-ub1}: For every $\theta>0$, it holds that
\begin{displaymath}
\lim_{\epsilon\rightarrow0}\limsup_{N\rightarrow\infty}{\frac{1}{N}}\ln\P[\lambda_{N\ell}(L_zL_z^*)\leq\epsilon] \leq\sup_\Psi[F(\theta,\Psi)-\widetilde{\mathcal{I}}(\Psi)-\ell J_{{\nu_z}}(\theta)]\, .
\end{displaymath}
We finally can let $\theta$ go to infinity to get 
$$%\begin{equation}\label{eq-ldub}
\lim_{\epsilon\rightarrow0}\limsup_{N\rightarrow\infty}\frac{1}{ N}\ln\P[\lambda_{N\ell}(L_zL_z^*)\leq\epsilon] \leq \liminf_{\theta\rightarrow\infty}\sup_\Psi[F(\theta,\Psi)-\widetilde{\mathcal{I}}(\Psi)-\ell J_{\nu_z}(\theta)]\, .
$$

\section{Proof of Theorem~\ref{thm-Lbounds}: Lower Bound}\label{sect-LB}
As before, we restrict the argument to $\beta=1$ and refer to Section~\ref{sect-betas} for $\beta=2$. Recall the set $\Omega_{\epsilon,\alpha,\kappa}$ from~\eqref{eq-defOm}. Similar to the upper bound, we set up by applying the tilt
\begin{align*}
{\frac{1}{ N}}\ln&\P(L_z^{(N)}\in\Omega_{\epsilon,\alpha,\kappa})	= {\frac{1}{N}} \ln \E\Biggl[ \mathds{1}_{\{L_z^{(N)}\in\Omega_{\epsilon,\alpha,\kappa}\}} \frac{I_{N\ell}(-\theta,L_z^{(N)}(L_z^{(N)})^*)}{I_{N\ell}(-\theta,L_z^{(N)}(L_z^{(N)})^*)}\Biggr]\, .\\
\end{align*}
Since we {aim to }take $\theta$ to $+\infty$ later, we need to upper bound the spherical integral in the denominator with some error. However, if we do so naively, we would end up with an error term of order $\varepsilon \theta${. This is insufficient, as we want to be able to choose $\theta$ large enough while keeping $\varepsilon$ fixed}. To solve this problem, we exploit the fact that $L_z^{(N)}(L_z^{(N)})^*$ has to be a positive matrix. 
Let $H_z$ be the self-adjoint matrix with the same eigenvalues and eigenvectors as $L_z^{(N)}(L_z^{(N)})^*$ except for the smallest one which is put to zero. Since $L_z^{(N)}(L_z^{(N)})^*$ is non-negative, $L_z^{(N)}(L_z^{(N)})^*\ge  H_z$ so that since $\theta\ge 0$, $$I_{N\ell}(-\theta,L_z^{(N)}(L_z^{(N)})^*)\le I_{N\ell}(-\theta,H_z)\,.$$
As a consequence, we deduce that
\begin{align}
\frac{1}{N}\ln&\P(L_z^{(N)}\in\Omega_{\epsilon,\alpha,\kappa})	\ge  \frac{1}{N} \ln \E\Biggl[ \mathds{1}_{\{L_z^{(N)}\in\Omega_{\epsilon,\alpha,\kappa}\}} \frac{I_{N\ell}(-\theta,L_z^{(N)}(L_z^{(N)})^*)}{I_{N\ell}(-\theta,H_z)}\Biggr]\nonumber\\
&\ge   \frac{1}{N} \ln \E\Biggl[ \mathds{1}_{\{L_z^{(N)}\in\Omega_{\epsilon,\alpha,\kappa}\}} {I_{N\ell}(-\theta,L_z^{(N)}(L_z^{(N)})^*)}\Biggr]-\ell J_{{\nu_z}}(\theta)-h_{N,\alpha}(\theta)\nonumber\\
&=\frac{1}{N} \ln\Biggl(\frac{\E[ \mathds{1}_{\{L_z^{(N)}\in\Omega_{\epsilon,\alpha,\kappa}\}}I_{N\ell}(-\theta,L_z^{(N)}(L_z^{(N)})^*)]}{\E[I_{N\ell}(-\theta,L_z^{(N)}(L_z^{(N)})^*)]}\Biggr)+\frac{1}{N}\ln\E[I_{N\ell}(-\theta,L_z^{(N)}(L_z^{(N)})^*)] \nonumber\\
&\qquad - \ell J_{{\nu_z}}(\theta)-h_{N,\alpha}(\theta)\, ,
\label{ldlbeq}\end{align}
where

\[ h_{N,\alpha}(\theta) = \sup_{H \in \Omega_{0,\alpha,\kappa}}\left| \frac{1}{N } \ln I_{N \ell} (- \theta,H)  - \ell J_{{\nu_z}}(\theta) \right|. \] 
By Lemma~\ref{lem-convsphint}, $h_{N,\alpha}(\theta)$ is such that for all $\theta \geq 0$

\[ \lim_{N \to \infty} h_{N,\alpha}(\theta) =0\, . \]

We also recognize the annealed spherical integral on the right-hand side of~\eqref{ldlbeq}. For a matching lower bound we thus need to find a good $\theta$ such that the first term on the right-hand side of \eqref{ldlbeq} vanishes in the limit $N\rightarrow\infty$. 
To get there, rewrite
\begin{align*}
&\frac{\E_W[ \mathds{1}_{\{L_z^{(N)}\in\Omega_{\epsilon,\alpha,\kappa}\}}I_{N\ell}(-\theta,L_z^{(N)}(L_z^{(N)})^*)]}{\E_W[I_{N\ell}(-\theta,L_z^{(N)}(L_z^{(N)})^*)]}=\frac{\E_u[\P^{(\mathbf{u},\theta)}(L_z^{(N)}\in\Omega_{\epsilon,\alpha,\kappa})\E_W[\exp(-{\frac{1}{2}}\theta N\ell\|L_z^{(N)}u\|_2^2)]}{\E_W[I_{N\ell}(-\theta,L_z^{(N)}(L_z^{(N)})^*)]}\\
&\qquad\qquad\qquad\qquad\qquad\qquad\qquad\qquad=1 -\frac{\E_u[\P^{(\mathbf{u},\theta)}(L_z^{(N)}\in\Omega_{\epsilon,\alpha,\kappa}^c)\E_W[\exp(-{\frac{1}{2}}\theta N\ell\|L_z^{(N)}u\|_2^2)]}{\E_W[I_{N\ell}(-\theta,L_z^{(N)}(L_z^{(N)})^*)]}
\end{align*}
with the tilted measure
\begin{displaymath}
\P^{(\mathbf{u},\theta)}(A)=\frac{\E_W[\I_A\exp(-{\frac{1}{2}}\theta N\ell\|L_z^{(N)}u\|_2^2]}{\E_W[\exp(-{\frac{1}{2}}\theta N\ell\|L_z^{(N)}u\|_2^2)]}\,.
\end{displaymath}
Our goal is to show that there is a suitable $\theta_0>0$ such that  for $\varepsilon>0$ going to zero
$$
\limsup_{N\rightarrow\infty}
\frac{\E_u[\P^{(\mathbf{u},\theta_0)}(L_z^{(N)}\in\Omega_{\epsilon,\alpha,\kappa}^c)\E_W[\exp(-{\frac{1}{2}}\theta_0 N\ell\|L_z^{(N)}u\|_2^2)]}{\E_W[I_{N\ell}(-\theta_0,L_z^{(N)}(L_z^{(N)})^*)]}\le 1/2\,.$$

By using Lemmas~\ref{lem-prelimbounds}(iii) and~\ref{lem-conc}, 
we see that for every $\theta>0$ we can choose $\kappa>0$ small enough and $\alpha$ large enough, to see that uniformly on $u$ on the sphere
$$%\begin{equation}\label{toprove2}
\limsup_{N\rightarrow\infty}
\P^{(\mathbf{u},\theta_0)}(L_z^{(N)}\in\Omega_{\alpha,\kappa}^c)\le 1/4.$$%\end{equation}

Hence,  our main task is to show that
\begin{equation}\label{toprove3}
\limsup_{N\rightarrow\infty}
\frac{\E_u[\P^{(\mathbf{u},\theta_0)}(\lambda_{N\ell}(L_z L_z^*)\ge \epsilon)\E_W[\exp(-{\frac{1}{2}}\theta_0 N\ell\|L_z^{(N)}u\|_2^2)]]}{\E_W[I_{N\ell}(-\theta_0,L_z^{(N)}(L_z^{(N)})^*)]}\le 1/4\, .
\end{equation}

To get there, let us note the following fact about Gaussian random variables under a tilted measure.
\begin{lemma}\label{lem-tiltedGaussian}
Let $Y$ be a real Gaussian vector with mean $\xi_Y$ and covariance matrix $\Sigma_Y$. Consider the measure $\P^{\mathrm{tilt}}$ that is tilted by $\exp(-\theta\|Y\|_2^2)/Z$, where $Z$ denotes the expectation of $\exp(-\theta\|Y\|_2^2)$ under the original measure $\P$. Then the law of $Y$ under $\P^{\mathrm{tilt}}$ is again Gaussian, and we note the formulas
\begin{equation}\label{eq-tiltedGaussian}
\xi_Y':=(\Id_{N\ell}+2\theta\Sigma_Y)^{-1}\xi_Y,\quad \Sigma_Y':=\Sigma_Y^{1/2}(\Id_{N\ell}+2\theta\Sigma_Y)^{-1}\Sigma_Y^{1/2}
\end{equation}
for the mean and the covariance matrix, respectively.
\end{lemma}

\begin{proof}
We start by observing that the law of $Y$ under $\P^{\mathrm{tilt}}$ has the density
\begin{displaymath}
f_{Y'}=\frac{1}{Z}\frac{1}{(2\pi)^{N\ell/2}|\Sigma|^{1/2}}\exp\Big(-\frac{1}{2}(x-\xi_Y)^T\Sigma_Y^{-1}(x-\xi_Y)-\theta x^Tx\Big)\, .
\end{displaymath}
From here, we extract
\begin{displaymath}
x^T\Big(-\theta\Id-\frac{1}{2}\Sigma_Y^{-1}\Big)x=-\frac{1}{2}x^T(\Sigma_Y')^{-1}x
\end{displaymath}
with $\Sigma_Y'$ from~\eqref{eq-tiltedGaussian}. Completing the square then yields
\begin{displaymath}
-\frac{1}{2}(x-\xi_Y)^T\Sigma_Y^{-1}(x-\xi_Y)-\theta x^Tx=-\frac{1}{2}(x-\xi_Y')^T(\Sigma_Y')^{-1}(x-\xi_Y')+\mathcal{R}\, ,
\end{displaymath}
where we collected all remaining deterministic terms in~$\mathcal{R}$. Computing $Z$ and combining the result with the above now yields that $f_{Y'}$ is indeed the density of a multivariate Gaussian with the desired parameters.
\end{proof}

Recalling that $\xi=\E[L_z^{(N)}u]=(A_0\otimes \Id_N)u$, applying Lemma~\ref{lem-tiltedGaussian}  to the Gaussian vector $Y=L_zu$ yields that the expectation of $\|L_z^{(N)}u\|_2^2$ under $\P^{(u,\theta)}$ is given by
\begin{equation}\label{eq-tiltedex}
\E^{(u,\theta)}[\|L_z^{(N)}u\|_2^2]=\frac{1}{N}\Tr[\Sigma(\Id_{N\ell}+\theta\ell\Sigma)^{-1}]+\|(\Id_{N\ell}+\theta\ell\Sigma)^{-1}\xi\|_2^2\, ,
\end{equation}
where we set again $\xi=\E[L_z^{(N)}u]$ and $\Sigma=N\Cov[L_z^{(N)}u]$. 

\medskip
\underline{\smash{Case $A_0=0$}:} If we set $A_0=0$,~\eqref{eq-tiltedex} simplifies to
\begin{displaymath}
\E^{(u,\theta)}[\|L_z^{(N)}u\|_2^2]=\frac{1}{N}\Tr[\Sigma(\Id_{N\ell}+\theta\ell\Sigma)^{-1}]\, .
\end{displaymath}
Clearly, $\E^{(u,\theta)}[\|L_z^{(N)}u\|_2^2]\to0$ as $\theta\to\infty$, which suggests that $\lambda_{N\ell}(L_zL_z^*)=\min_{\|u\|_2=1}u^TL_z^{(N)}(L_z^{(N)})^*u$ being small becomes typical. Indeed, by Markov's inequality, we get for every $u$ in the sphere that
\begin{align*}
\P^{(u,\theta)}(\lambda_{N\ell}(L_zL_z^*)>\epsilon)&\leq\P^{(u,\theta)}(\|L_z^{(N)}u\|_2^2>\epsilon)\leq\frac{\E^{(u,\theta)}[\|L_z^{(N)}u\|_2^2]}{\epsilon}\\
&=\frac{1}{N\epsilon}\Tr[\Sigma(\Id_{N\ell}+\theta\ell\Sigma)^{-1}]\le \frac{1}{ \epsilon\theta}\, . \nonumber
\end{align*}
The right-hand side is smaller than $1/4$ if $\epsilon\theta \ge 4$, which we can assume.

\medskip
\underline{\smash{Case $A_0\neq 0$}:} Consider now~\eqref{eq-tiltedex} in its general form. As a tool, define
\begin{displaymath}
\sigma^N_0 := \min_{M_1,\dots,M_{k} \in\mathrm{Sym}_N(\R)} \Bigl\{ \sigma_{\min}\Bigl(\sum_{j=1}^k A_j \otimes M_j + A_0 \otimes \Id_N\Bigr)^2 \Bigr\}\, .
\end{displaymath}
We first show that the rate function is infinite when this minimal singular value is positive. As a first step, we note that this quantity does not depend on $N\ge \ell$. 

\begin{lemma}\label{lem-claim1}
We have $\sigma^N_0=\sigma^\ell_0$ for every $N\ge \ell$.
\end{lemma}

\begin{proof}
By definition, we have for every $v \geq 0$ that
\[ \sigma_0^{N} \leq v \Leftrightarrow  \exists M_1,\ldots,M_k\in \mathrm{Sym}_N(\R),\,  \exists u \in \mathds{S}^{N \ell -1}:\ \Big\| \Bigl(\sum_{k} A_k \otimes M_k + A_0 \otimes \Id_N\Bigr) u\Big\|_2^2 \leq v\, .\]
For every $u, M_1, \dots M_k$ satisfying the above inequality, one can choose an orthogonal matrix $O$ that maps the components $u_1,\dots, u_{\ell}$ arbitrarily to the subspace $V^{N}_{\ell}\subset \R^N$ spanned by the first $\ell$ vectors of the canonical basis (i.e., the vectors whose coordinates $\ell+1$ to $N$ vanish). If we call $\Pi$ the linear application from $\R^N$ to $\R^{\ell}$ that maps $(x_i)_{i \leq N}$ to $(x_i)_{i \leq\ell}$, taking $\widetilde{M}_k := \Pi O M_k O^{-1} \Pi^{t} $, and $\widetilde{u} = \Id_{\ell} \otimes (\Pi O)u$, we can see that 

\begin{displaymath}
\Big\| \Bigl(\sum_{j=1}^k A_j \otimes \widetilde{M}_j + A_0 \otimes \Id_\ell\Bigr) \widetilde{u}\Big\|_2^2 = \Big\| \Bigl(\sum_{j=1}^k A_j \otimes M_j + A_0 \otimes \Id_N\Bigr) u\Bigr\|_2^2\, .
\end{displaymath}
In particular $\Pi^t$ is the linear application that takes a vector of $\R^{\ell}$ and ``completes'' it to a vector of $\R^N$ by setting all coordinates of indices $\ell+1$ to $N$ to zero. The vector $\widetilde{u}$ on the left hand side is essentially a unit vector up to projection and $\widetilde{M}_k$ are matrices of $\mathrm{Sym}_\ell(\R)$. Therefore, $\sigma^{\ell}_0 \leq \sigma^N_0$, and the equality is proven as the reverse inequality is trivial.
\end{proof}

Following Lemma~\ref{lem-claim1}, we abbreviate $\sigma^N_0=\sigma_0$. We claim that the 
rate function is infinite when this minimal singular value is positive.

\begin{lemma}\label{lem-claim2}Assume that $\sigma_0>0$, then
$$-\mathcal{I}_{L_z}=
\lim_{\theta\to\infty}\sup_\Psi[F(\theta,\Psi)-\widetilde{\mathcal{I}}(\Psi)-\ell J_{\nu_z}(\theta)]=-\infty\, .$$
Recall that $F$ was introduced in~\eqref{eq-defF}, $\widetilde{\mathcal{I}}$ is the rate function in Lemma~\ref{lem-WishartLDP}, and $J_{\nu_z}$ is the limit of the spherical integral (cf. Lemma~\ref{lem-convsphint}).
\end{lemma}
\begin{proof}
Recall that
\[ \widetilde{F}(\theta) = \lim_{N \to \infty} \frac{1}{N}\ln \E[ I_{N \ell}(-\theta, L_z^{(N)}(L_z^{(N)})^*)] = \max_{\Psi} (F(\theta,\Psi) - \widetilde{\mathcal{I}}(\Psi))\,. \]
Clearly, $\widetilde{F}( \theta ) \leq -\frac{\theta }{2}\ell\sigma_0$ since $\|L_z^{(N)} u\|^2\ge \sigma_0$ for every vector $u$ and all matrices $(W_1,\ldots,W_k)$. Therefore,
$$-\mathcal{I}_{L_z^{(N)}}\le 
\lim_{\theta\to\infty}\sup_\Psi[-\tfrac{\theta}{2}\ell\sigma_0-\widetilde{\mathcal{I}}(\Psi)-\ell J_{\nu_z}(\theta)]\, .$$
However, by the definition of $J_{\nu_z}$ given in \eqref{eq-defJ}, $J_{\nu_z}(\theta)$ goes to infinity more slowly than~$\theta$. More precisely, it behaves like $\ln \theta$, as $\theta$ goes to infinity. The result follows. 
\end{proof}

As a consequence, we may restrict ourselves to prove the lower bound when $\sigma_0=0$.

\begin{lemma}\label{lem-claim3}
The annealed spherical integral is asymptotically equivalent to $e^{-N\frac{\theta}{2}\sigma_0 \ell}$, i.e.,
$$
\lim_{\theta \to \infty}\frac{\widetilde F(\theta)}{ \theta} = - \frac{1}{2} \ell\sigma_0\,.
$$
\end{lemma}

\begin{proof}
We already discussed  the easy bound $\widetilde{F}( \theta ) \leq - \frac{\theta}{2} \ell\sigma_0$. To get the converse bound, it suffices to prove that for all $\epsilon >0$, there exists some $c(\epsilon)>0$ finite  such that 
\begin{equation}\label{bing}
\P[ \|L_z^{(N)} u \|_2^2 \leq \sigma_0 + \epsilon] \geq e^{ -N (c(\epsilon)+o(1))}\,,
\end{equation}
where the probability holds on both $u$ and $W_j^{(N)}, 1\le j\le k$. Indeed from such a bound, one gets by localizing the expectation 
$$\E_W[I_{N\ell}(-\theta,L_zL_z^*)]
\ge \E_W[\E_u[\I_{\|L_z^{(N)} u \|_2^2 \leq \sigma_0 + \epsilon}e^{-\frac{\theta}{2}\ell N\|L_z u\|^2_2}
]]\ge e^{ N(-\frac{\theta}{2} (\sigma_0 +\epsilon)\ell  - c( \epsilon)+o(1))}$$ from which we deduce
$\widetilde F ( \theta) \geq - \frac{\theta}{2}\ell (\sigma_0 +\epsilon)  - c( \epsilon) $. 
Now to get the lower bound \eqref{bing}, it is enough to notice that
\begin{displaymath}
\|L_z^{(N)} u \|_2^2=\sum_{j,j'=0}^k\sum_{s,t=1}^\ell(A_j^TA_{j'})_{s,t}\langle W^{(N)}_j u_s,
W^{(N)}_{j'} u_t\rangle
\end{displaymath}
with the convention that $W^{(N)}_0= \Id_N$ so that to achieve the desired deviation \eqref{bing}, we only need the matrix $\big(\langle W^{(N)}_j u_s,
W^{(N)}_{j'} u_t\rangle, j,j'\le k, s,t\le \ell\big)$ to be $\delta(\epsilon)$-close to the matrix $\big( \langle \widetilde M_j \widetilde u_s, \widetilde M_{j'}\tilde u_t\rangle\big)$ where $(\widetilde M_1,\ldots, \widetilde M_k,\widetilde u)$ are minimizers in $\sigma_0^\ell$ and $\delta(\epsilon)>0$ is small enough. 

\medskip
If we take the same notation $\Pi$ as in the proof of Lemma~\ref{lem-claim1} and denote by $\Delta$ the application from $\R^N$ to $\R^{N - \ell}$ that maps $(x_i)_{i \geq 1}$ to $(x_i)_{i \geq \ell+1}$. Let
\begin{align*}
\hat{u}_i &:= \Pi^{t} \widetilde{u}_i, \quad i=1,\dots,\ell\\
\widetilde{W}_j^{(N)} &:= \Pi W_j^{(N)}  \Pi^t, \quad j = 0, \dots, k\\
Y_j^{(N)} &:= \Delta W_j^{(N)} \Pi^t, \quad j = 0, \dots, k.
\end{align*}
Observe that $\widetilde{W}_0^{(N)}=\Id_{\ell}$ and $Y_0^N = 0$. We further have
\begin{displaymath}
\langle W^{(N)}_j \hat{u}_s,
W^{(N)}_{j'} \hat{u}_t\rangle = \langle \widetilde{W}^{(N)}_j  \widetilde{u}_s,
\widetilde{W}^{(N)}_{j'} \widetilde{u}_t\rangle + \langle Y^{(N)}_j \widetilde{u}_s,
Y^{(N)}_{j'} \widetilde{u}_t\rangle\, .
\end{displaymath}
Therefore, we only need the matrices $\widetilde{W}^{(N)}_j$ to be close to $\widetilde{M}_j$ and the matrices $Y_j^N$ to be close in operator norm to $0$, provided $u-\hat u$ is small. More precisely, we find $\delta(\varepsilon)$ going to zero with $\varepsilon$ such that
\begin{align}
 \Bigl(\bigcap_j \{\|\widetilde{W}^{(N)}_j&-  \widetilde{M}_j\|_{\rm{op}}\le \delta(\varepsilon)\}\Bigr)\cap %\Bigl(\bigcap_t\{\|\widetilde{u}_t-\widetilde{v}_t\|_2\le \delta(\varepsilon)\}\Bigr)\cap
 \Bigl(\bigcap_j\{\|Y_j^N\|_{\rm{op}}\le \delta(\varepsilon)\}\Bigr)\nonumber\\
  &\cap \{\|L_z^{(N)}\|_{\rm{op}}\le M\}\cap \{\|u-\hat u\|_2\le \varepsilon/M\}\subset  \{\|L_z^{(N)}u\|_2^2\le \sigma_0+\varepsilon\}\, .\label{zxc} 
\end{align}
The two last sets ensure that $\|L_z^{(N)}(u-\hat u)\|_2$ is small. The matrix $\widetilde{W}^{(N)}_j$ is distributed like the $\ell \times \ell$ top left corner of an $N \times N$ GOE matrix. Its distribution has the density 

\[\frac{1}{Z_N} \exp \Big( - \frac{N  \Tr ( |H|^2)}{4} \Big) \]

with respect to the Lebesgue measure on the $\ell \times \ell$ real symmetric matrices. Therefore, using a classical Laplace method, we find 

\[ \P[ \| \widetilde{W}_j^N - \widetilde{M}_j \|_{\rm{op}} \leq \eta ] \geq  \exp\Bigl( -\frac{N (\Tr(| \widetilde{M}_j|^2) + o(1))}{4} \Bigr)\,, \] 
where $o(1)$ goes to zero when $\eta$ goes to zero.
For the matrices $Y_j^N$ we will use that we can bound $||Y_j^N||_{\rm{op}}$ by $\sqrt{ \Tr(Y_j^N (Y_j^N)^*)}$. Note, however, that $\Tr(Y_j^N (Y_j^N)^*)$ is a Gamma random variable with shape parameter $\frac{(N-\ell) \ell}{2}$ and scale parameter $2/N $. Once again, with a classical Laplace method, one can see that this random variable satisfies on $\R^+$ a large deviation principle with speed $N$ and rate function $x \mapsto \frac{1}{2}[ (x- \ell) - \ell \ln (x/ \ell)]$. Therefore, we find  that 

\begin{eqnarray*} \P[ \|Y_j^N \|_{\rm{op}} \leq \eta ] &\geq&  \P[\Tr(Y_j^N (Y_j^N)^*)\leq \eta^2 ] \\
&\geq& \P \Big[ \left| \Tr(Y_j^N (Y_j^N)^*) - \frac{\eta^2}{2} \right| \leq \frac{\eta^2}{2} \Big] \\
&\geq& \exp\Bigr(- \frac{N }{2} \Big(\frac{\eta^2}{2} - \ell   -\ell \ln \frac{\eta^2}{2 \ell}    +o(1)\Big)  \Bigl)\, , \end{eqnarray*} 
where $o(1)$ goes to zero when $N$ goes to infinity and then $\eta$ goes to zero.
Lastly using classical neighborhood estimation on the unit sphere, we can write for $u$ distributed uniformly on the sphere $\mathds{S}^{N\ell-1}$
\begin{equation}\label{eq-uminusuhat}
\P[\| u - \hat{u}\|_2 \leq \varepsilon ] \geq e^{ N( \ell \ln \varepsilon +o(1))}\, ,
\end{equation}
where $o(1)$ goes to zero when $N$ goes to infinity and then $\varepsilon$ goes to zero.
Finally, since the~$\tilde{W}_j^N$,~$Y_j^N$, and~$u$ are independent random variables, we may write by \eqref{zxc}, for some $\eta >0$ depending on $\epsilon/2$
\begin{align*} 
&\P[ \|L_z^{(N)} u \|^2 \leq \sigma_0 + \varepsilon] + \P[ \|L_z^{(N)}\|_{\rm{op}} \geq M ] \\
&\quad\geq \P[ \forall j=1,\dots, k,  \|Y_j^N \|_{\rm{op}} \leq \eta ]\ \P[ \forall j=1,\dots, k,  \|\tilde W_j^N - \tilde{M}_j\|_{\rm{op}} \leq \eta ]\ \P[  \|u - \hat{u}\|_2 \leq \varepsilon/2M ]\, .
\end{align*}
If we take $\frac{1}{N}\ln$ and then the  $\liminf$ of the left hand side, we deduce that 

\[ \liminf_{N\to\infty} \frac{1}{N} \ln \Big(  \P[ \|L_z^{(N)} u \|^2 \leq \sigma_0 + \varepsilon] + \P[ \|L_z^{(N)}\|_{\rm{op}} \geq M ] \Big) \geq - c( \varepsilon) + \ell (\ln(\varepsilon) - \ln(2M))  \]
for some $c(\epsilon)>0$. Recalling Lemma \ref{lem-prelimbounds}(iii) and choosing $M=M(\varepsilon) >0$ large enough such that $c'(\varepsilon)=c( \epsilon) - \ell (\ln(\epsilon) -\ln(2M)) < \frac{1}{4}\ell\sqrt{M}-c$, we can write that 

\[  \liminf_{N\to\infty} \frac{1}{N} \ln  \P[ \|L_z^{(N)} u \|_2^2 \leq \sigma_0 + \varepsilon]  \geq - c'(\varepsilon)  \]
for some $c'(\varepsilon) >0$.
\end{proof}

\medskip
We finally prove \eqref{toprove3}. We denote by 
\begin{eqnarray*}
\P^{(\theta)} (A)&:=&
\frac{\E_u[\P^{(\mathbf{u},\theta)}(L_z^{(N)}\in A)\E_W[\exp(-\frac{\theta}{2} N\ell\|L_z^{(N)}u\|_2^2)]]}{\E_W[I_{N\ell}(-\theta,L_z^{(N)}(L_z^{(N)})^*)]}\\
&=&\frac{\E_W[\I_{\{L_z^{(N)}\in A\}} I_{N\ell}(-\theta,L_z^{(N)}(L_z^{(N)})^*)] }{\E_W[I_{N\ell}(-\theta,L_z^{(N)}(L_z^{(N)})^*)]}\,,\\
\end{eqnarray*}
and $\E^{(\theta)}$ for the associated expectation.
\begin{lemma}
Assume that $\sigma_0=0$. Then,
for every $\epsilon>0$
$$\limsup_{\theta\rightarrow \infty }\limsup_{N\rightarrow \infty }\P^{(\theta)}(\lambda_{N\ell}(L_zL_z^*)\ge  \epsilon )=0\,.$$
\end{lemma}
\begin{proof}
By Markov's inequality, it is enough to prove that 
$$\lim_{\theta\rightarrow \infty }\limsup_{N\rightarrow \infty }\E^{(\theta)}[ \|L_z^{(N)}u\|_2^2]\le \sigma_0=0\,.$$
But
\[ \E^{(\theta)}[ \|L_z^{(N)}u\|^2] = - \frac{2}{N\ell}\frac{\partial}{ \partial \theta}  \ln \E[ I_{N \ell}(-\theta,L_z^{(N)} (L_z^{(N)})^*)] \]
and, therefore, by convexity of $\theta\mapsto \ln \E[I_{N \ell}(-\theta,L_z^{(N)} (L_z^{(N)})^*)] $, using the slope inequalities, we have 
\[\frac{\ell}{2}\E^{(\theta)}[ \|L_z^{(N)}u\|_2^2] \leq  \frac{1}{N}  \ln \E[ I_{N \ell}(-(\theta-1),L_z^{(N)} (L_z^{(N)})^*)]-  \frac{1}{N}  \ln \E[ I_{N \ell}(-\theta,L_z^{(N)} (L_z^{(N)})^*)] \,. \]
Thus, taking $N$ to infinity we get 

\[ \limsup_{N \to \infty}\frac{\ell}{2} \E^{(\theta)}[ \|L_z^{(N)}u\|_2^2] \leq   \widetilde{F}(\theta-1) - \tilde{F}(\theta)\, .\]

Now one readily checks that $\widetilde{F}$ is convex (as a limit of convex functions) therefore, if we denote $\widetilde{F}'_r$ its right  derivative and $\widetilde{F}'_l$ its left derivative, both are increasing and by the bounds given above, we have necessarily that they must go to $- \frac{\ell}{2} \sigma_0$ for $\theta$ going to $+ \infty$. We therefore deduce
\[ \limsup_{N \to \infty}\frac{\ell}{2} \E^{(\theta)}[ \|L_z^{(N)}u\|_2^2] 
 \leq - \widetilde{F}_r'(\theta-1)\, . \]
Taking $\theta$ to $\infty$ gives
\[ \lim_{\theta \to \infty} \limsup_{N \to \infty} \E^{(\theta)}[ \|L_z^{(N)}u\|_2^2] \le \sigma_0\,. \]
\end{proof}

\begin{remark}
We can generalize the previous arguments to the case where the profile $\rho(u)$ of $u$ is fixed:
Define
$$\sigma^N_0(\Psi) := \min_{M_1,\dots,M_{k} \in\mathrm{Sym}_N(\R)} \min_{u\in \mathds{S}^{N\ell-1}:
\rho(u)=\Psi} \Bigl\{ \Big\| \Bigl(\sum_{k} A_k \otimes M_k + A_0 \otimes \Id_N\Bigr) u\Big\|_2^2\Bigr\}\, . $$
Then, $\sigma^N_0(\Psi)$ does not depend on $N$ large enough and if we denote by $\sigma_0(\Psi)$ its value for~$N$ large enough, 
$$-\mathcal{I}_{L_z}:= \lim_{\theta\to\infty}\sup_{\Psi:\sigma_0(\Psi)\le \epsilon(\theta)} [F(\theta,\Psi)-\widetilde{\mathcal{I}}(\Psi)-\ell J_{\nu_z}(\theta)]\, ,$$
where $\epsilon(\theta)$ goes to zero with $\theta$.
Indeed, similarly for every $u$ so that $\rho(u)=\Psi$,
$$\E_W[\exp(-\tfrac{\theta}{2} N\ell \|L_z^{(N)} u\|^2)] \le \exp(-\tfrac{\theta}{2} N\ell  \sigma_0(\Psi))$$
from which we see that 
$$F(\theta,\Psi)-\widetilde{\mathcal{I}}(\Psi)-\ell J_{\nu_z}(\theta)\le -\frac{\ell \theta}{2} \sigma_0(\Psi)+\frac{\ell}{2}\ln \theta +C\, ,$$
where $C$ is some finite constant. Thus, if $\sigma_0(\Psi)\gg \ln\theta /\theta$, the right hand side goes to~$-\infty$. The result follows.

\end{remark}

\section{Proof of Theorem \ref{thm-main}}\label{Sec:proofthm}
As before, let us fix $\beta=1$, and refer to Section~\ref{sect-betas} for the modification of the argument to $\beta=2$. We first recall the convergence of the spectrum of $P_{N}=P( W_1^{(N)}, \cdots, W_k^{(N)})$. By Voicu\-lescu's result, see \cite[Thm.~5.1.1]{AGZBook}, the distribution of the empirical measure $\hat\mu_{P_N}$ of the eigenvalues of $P(W^{(N)}_1,\ldots,W^{(N)}_k)$ converges towards the distribution $\mu_P$ of $P(S_1,\ldots,S_k)$ where $S_1,\ldots,S_k$ are free semi-circular variables.  Moreover, by concentration of measure, we can see
by adapting the argument of \cite[Cor.~6.12]{AliceStFlour} that for every Lipschitz function~$f$, every $\delta>0$
\begin{equation}\label{expti}
\limsup_{N\rightarrow\infty}\frac{1}{N}\ln\P\Bigl(\Big|\int f\mathrm{d}\hat\mu_{P_N}-\E\Big[\int f \mathrm{d}\hat\mu_{P_N}\Big]\Big|>\delta\Bigr)=-\infty\,.
\end{equation}
In fact, \cite[Cor.~6.12]{AliceStFlour} and \cite[Thm.~6.6.2]{AliceStFlour} cover only the case $f(x)=x$ (take $V=0$). However,  the argument generalizes as follows. First, consider
\begin{displaymath}
\Lambda_M^N:=\bigcap_{1\le i\le k}\{X_i\in \mathrm{Sym}_N(\R):\ \|X_i\|_{\rm{op}}\le M\}\,.
\end{displaymath}
We claim that on $\Lambda_M^N$,
$\int f\rm{d}\hat\mu_{P_N}$ is a Lipschitz function of the entries of the matrices $(W_1^{(N)},\ldots, W_k^{(N)})$ with some Lipschitz constant $C(M)\|f\|_L/\sqrt{N}$.
Indeed,
\begin{eqnarray}
    &&\Big|\frac{1}{N}\Tr(f(P(X_1,\ldots, X_k))-\frac{1}{N}\Tr(f(P(Y_1,\ldots, Y_k))\Big|\nonumber\\
    &\le& \|f\|_L\frac{1}{N}\Tr(|P(X_1,\ldots, X_k)-P(Y_1,\ldots, Y_k)|)\nonumber\\
    %&\le& \|f\|_L\frac{1}{N}\Tr(|P(X_1,\ldots, X_k)-P(Y_1,\ldots, Y_k)|)\nonumber\\
    &\le& c(M)\|f\|_L\sum_{i=1}^k\frac{1}{N}\Tr(|X_i-Y_i|)\le \frac{c(M)\sqrt{k}}{\sqrt{N} }\sqrt{\sum_{i=1}^k \Tr|X_i-Y_i|^2}\,.\nonumber\end{eqnarray}
Since $(W_1^{(N)},\ldots, W^{(N)}_k)$ have Gaussian entries with variance of order $1/N$, we deduce (see, e.g., Lemma 4.2, Corollary 4.7, and Lemma 5.9 of \cite{AliceStFlour}) that
\begin{equation*}
\P\Bigl( \Bigl\{\Big|\int f\rm{d}\hat\mu_{P_N}-\E\Bigl[\int f\rm{d}\hat\mu_{P_N}\Bigr]\Big|>\delta+\delta_M^N\Bigr\}\cap \Lambda_M^N\Bigr)
\le e^{-c\frac{N^2\delta^2}{C(M)^2}}
\end{equation*}
with $\delta_M^N=\E[ 1_{(\Lambda^N_M)^c} (c_1(M)+c_2(M)\sqrt{\sum_i \Tr (W_i^{(N)})^2})]$. 
\eqref{exptlambda} implies that $\delta_M^N $ goes to zero for $M\ge M_0$ and also \begin{equation}\label{exptih}
\P((\Lambda_M^N)^c)\le k e^{-\frac{\beta}{8} N M^2}\,.
\end{equation} 
With~\eqref{exptih}, we get \eqref{expti}. 

\medskip
By a covering argument and \eqref{expti},
we deduce as in \cite{GuionnetZeitouni00} that
\begin{equation}\label{convd}
\limsup_{N\rightarrow\infty}\frac{1}{N}\ln\P(\rm{d}(\hat\mu_{P_N},\mu_P)>\delta)=-\infty\,,
\end{equation}
where $\mathrm{d}(\cdot,\cdot)$ denotes a suitable metric compatible with the weak convergence of probability measures, e.g., the Wasserstein distance. Furthermore, by \cite{HST2006,Schultz2005} we know that $\lambda_1(P)$ converges almost surely towards $r_P$, the right boundary of the support of $\mu_P$ since it asserts that the spectrum of $P(W_1^{(N)},\ldots, W_k^{(N)})$ stays at distance smaller than $\delta$ from the spectrum of $P(S_1,\ldots,S_k)$ with large probability. Note that we can use the same argument  as above and \cite[Thm.~6.6.2]{AliceStFlour} to show that $\lambda_1(P)$ concentrates as the eigenvalues are Lipschitz functions of the entries of the matrix, namely for every $M>0$ there exist~$c, c(M)>0$, so that for every $ \delta>0$,
\begin{equation}\label{exptigl}
\P( \{|\lambda_1(P)
-\E[\lambda_1(P)]|>\delta+\delta_M^N\}\cap \Lambda_M^N)
\le e^{-c\frac{N\delta^2}{c(M)^2}}\, ,
\end{equation}
where $\delta_M^N$ goes to zero as $N$ goes to infinity for every $M$ large enough. The same holds for the smallest eigenvalue.

\medskip
Equipped with this tool, we now turn to the proof of Theorem~\ref{thm-main}. 
\begin{proof}[Proof of Theorem~\ref{thm-main}] We start with the large deviation principle for the spectrum given in Part (i).

\underline{\smash{Exponential tightness}}: We first observe that the distribution of the spectrum is exponentially tight since $P_{N}$ can only have an eigenvalue greater than $M$ if one of the matrices $W_i^{(N)}$ has norm larger than some $cM^\alpha+c'$ for some $c,\alpha>0$ and $c'$ finite.
We, therefore, deduce from \eqref{exptlambda} that exponential tightness holds in the sense that
$$\limsup_{M\to\infty}\limsup_{N\to\infty}\frac{1}{N}\ln \P(
\Spec(P_N)\cap [-M,M]^c\neq\emptyset)=-\infty\,.$$
As a consequence, it is enough to prove a weak large deviation principle, namely the theorem for small balls:
\begin{eqnarray}
&&\lim_{\varepsilon\to 0}\limsup_{N\to\infty}\frac{1}{N}\ln \P(
\Spec(P_N)\cap [x-\varepsilon,x+\varepsilon]\neq\emptyset)\qquad\nonumber\\
&&\qquad =\lim_{\varepsilon\to 0}\liminf_{N\to\infty}\frac{1}{N}\ln \P(
\Spec(P_N)\cap [x-\varepsilon,x+\varepsilon]\neq\emptyset)=-\mathcal{I}_{\Spec(P)}(x)\,.\label{wldpspec}\end{eqnarray}

\underline{\smash{Weak LDP for the spectrum}}: This is a direct consequence of Theorem \ref{thm-Lbounds} since
$$\{\Spec(P_N)\cap [x-\varepsilon,x+\varepsilon]\neq\emptyset\}=\{\dist(x,\Spec(P_N))\le \varepsilon\}$$ 
and hence, by Theorem \ref{thm-Lbounds}, we deduce that Theorem \ref{thm-main}(i) holds with $\mathcal{I}_{\Spec(P)}(x)=\mathcal I_{L_x}$.

\smallskip
\underline{\smash{Properties of the rate function}}: The fact that $\mathcal{I}_{\Spec(P)}$ is a good rate function is a direct consequence from the LDP, see e.g \cite[Lemma 1.2.18 (b)]{DemboZeitouniBook}. It can also be argued directly as follows. To show that $z\mapsto \mathcal{I}_{L_z}$ is lower semi-continuous,  note that by Lemma \ref{lem-prelimbounds}(iii), $L_z^{(N)}$ is bounded by $M$ with probability greater than $1-e^{(c-M\ell/4)N}$. On this set
$$|\lambda_{N\ell}(L_z (L_z)^*)-
\lambda_{N\ell}(L_{z'} L_{z'}^*)|\le 2M |z-z'|\,.$$
As a consequence, if $\mathcal{I}_{L_z}$ is finite, we can find $M$ and $N$ large enough so that
$$\P(|\lambda_{N\ell}(L_zL_z^*)|\le \varepsilon)\ge \P(|\lambda_{N\ell}(L_{z'}L_{z'}^*)|\le \delta)-e^{(c-M\ell/4)N}$$
for $\delta<\varepsilon/2$ and $|z-z'|\le \varepsilon/4M$.Taking liminf in $N$ on both sides, then $\delta$ to zero yields
$$\liminf_{N\to\infty}\frac{1}{N}\ln \P(|\lambda_{N\ell}(L_zL_z^*)|\le \varepsilon)\ge -\mathcal{I}_{L_{z'}} \,.$$
Finally, taking instead of $z'$ any sequence $z_n$ converging towards 
$z$, so as close to $z$ for $n$ large as wished, we deduce by letting $\varepsilon$ go to zero that
$$-\mathcal{I}_{L_z}\geq \limsup_{n\to\infty} -\mathcal{I}_{L_{z_n}}$$ 
which is equivalent to saying that $z\mapsto \mathcal{I}_{L_z}$ is lower semi-continuous. Thus its level sets are closed. They are included in compacts by exponential tightness.

\smallskip
We observe from the last part that $\mathcal{I}_{\Spec(P)}(x)=\mathcal I_{L_x}$. This shows in particular that $\mathcal I_{L_x}$ does not depend on the choice of the linearization.

\medskip
Next, we prove Part (ii) of Theorem~\ref{thm-main}.

\underline{\smash{Exponential tightness}}: The random variable $\lambda_1(P)$ is exponentially tight since to be large, one at least of the matrices needs to have a large norm, an event which is exponentially small by \eqref{exptih}. 

\smallskip
\underline{\smash{LDP for the largest eigenvalue}}: We first remark that for every $\varepsilon>0$,
$$\limsup_{N\to \infty}\frac{1}{N}\ln\P(\lambda_1(P)\le r_P-\varepsilon)=-\infty$$
by \eqref{convd} since on $\{\lambda_1(P)\le r_P-\varepsilon\}$ the empirical measure can not be close to $\mu_P$. As a consequence, 
we  have $\mathcal{I}_{\lambda_{\rm{max}}(P)}(x)=+\infty$ if $x<r_P$ .

\medskip
If $x>r_P$, we observe that
    $$%\begin{equation}\label{lkj}
    \P(
    \lambda_1(P)\ge x+\varepsilon)\le\P(\exists y\ge x:
    \dist(y, \Spec(P_N))\le \varepsilon)
    \le \P(
    \lambda_1(P)\ge x-\varepsilon)\,.
   $$% \end{equation}
    Noticing that
    $$\{\exists y\ge x:
    \dist(y, \Spec(P_N))\le \varepsilon\}=\{\Spec(P_N)\cap [x-\varepsilon,+\infty[\neq\emptyset\}\, ,$$
    we deduce from the first point of Theorem \ref{thm-main} that
    \begin{equation}\label{wldub}\limsup_{N\to\infty}\frac{1}{N}\ln \P(\lambda_1(P)\ge x)\le -\lim_{\varepsilon \downarrow 0}\inf_{z\ge x-\varepsilon } \mathcal{I}_{\Spec(P)}(z)=- \inf_{z\ge x} \mathcal{I}_{\Spec(P)}(z)\,,\end{equation}
    where we used that $\mathcal{I}_{\Spec(P)}(z)$ is a good rate function.
    Similarly,
    \begin{equation}\label{wldlb}
    \liminf_{N\to\infty}\frac{1}{N}\ln \P(\lambda_1(P)\ge x)\ge -\lim_{\varepsilon \downarrow 0}\inf_{{z> x+\varepsilon} } \mathcal{I}_{\Spec(P)}(z)=-\inf_{{z> x} } \mathcal{I}_{\Spec(P)}(z) \,.\end{equation}

\medskip
We next show that $ \mathcal{I}_{\Spec(P)}$ is increasing on $]r_P,+\infty[$. To this end, we follow the idea of~\cite[Lemma B.1]{BoursierG24} to interpolate the matrices.
    We may and shall assume that $\mathcal{I}_{\Spec(P)}(x)$ is finite since otherwise we are done.  We consider  the symmetric Ornstein-Uhlenbeck process on matrices 
$${\mathrm{d}G^{(N)}(s)=\mathrm{d}H^{(N)}(s)-\frac{1}{2} G^{(N)}(s) \mathrm{d}s}\, ,$$
where $H^{(N)}$ is a $N\times N$ symmetric Brownian motion and $G^{(N)}({0})=W^{(N)}$ is a GOE with variance $1$. 
We also denote by $\bar G^{(N)}(s)$ the solution constructed with the same Brownian motion but starting from zero. It is well known that $G^{(N)}(s)$ has the law of a Gaussian ensemble  with variance $1$ at any time $s$, whereas $\bar G^{(N)}(s)$ has  the law of the Gaussian ensemble with variance $(1- e^{-s})$.  Moreover,
$${\mathrm{d}(G^{(N)}(s)-\bar G^{(N)}(s))^{2}=-(G^{(N)}(s)-\bar G^{(N)}(s))^{2} \mathrm{d}s}$$
implies that
\begin{equation}\label{dis}
\|G^{(N)}(s)-\bar G^{(N)}(s)\|_{\rm{op}}\le \|W^{(N)}\|_{\rm{op}} e^{-s/2}\, .
\end{equation}
We will later take $s=N$. Moreover we work on the event where $\|W^{(N)}\|_{\rm{op}}\le N$ so that $\|G^{(N)}(s)-\bar G^{(N)}(s)\|_{\rm{op}}$ is of order $Ne^{-s/2}\le \delta/2$ if $s$ is of order $N$. %(here $t$ is fixed, independently of $N$).
We take $k$ independent copies $G^{(N)}_i, \bar G^{(N)}_i$ with independent initial data $W^{(N)}_i$, $1\le i\le k$, of these random processes. We denote in short $\|X^{(N)}\|_{\rm{op}}=\max_{1\le i\le k}\|X^{(N)}_i\|_{\rm{op}}$ for $k$ $N\times N$ Hermitian matrices. 
We finally set 
$$X^{(N)}(s):= P(G^{(N)}_1(s),\ldots,G^{(N)}_k(s))\mbox{ and }\bar X^{(N)}(s):=P(\bar G^{(N)}_1(s),\ldots,\bar G^{(N)}_k(s))\,.$$
Note that for every integer $p$,  $X^{(N)}(\frac{p}{N})$ has the same law as $X^{(N)}(0)$ and in particular its largest eigenvalue converges towards $r_P$ almost surely and in expectation. Observe also that by \eqref{dis}, there exist  finite constants $c,c'$ so that
\begin{equation*}
\|\bar X^{(N)}(N)-X^{(N)}(N)\|_{\rm{op}} \le c(\max_{1\le \ell\le k}\|G^{(N)}_{\ell}(N)\|^d_{\rm{op}} +
\|\bar G^{(N)}_\ell(N)\|^d_{\rm{op}}+c')
\|W^{(N)}\|_{\rm{op}}\, e^{-N/2}\, .\end{equation*}
As a consequence, since the moments of operator norms of Gaussian matrices are well controlled,  we find a finite constant $C$ such that
$$| \E[\lambda_{1}(\bar X^{(N)}(N))]-\E[\lambda_{1}( X^{(N)}(N))]|\le Ce^{-N/2}\, .$$
Hence, 
$\E[\lambda_{1}(\bar X^{(N)}(N))]$ converges as $N$ goes to infinity towards $r_P$. By concentration of measure for $\lambda_{1}(\bar X^{(N)}(N))$ {(see \cite[Thm.~2.3.5]{AGZBook})}, we also have that for every $\eta>0$, every $M$, there exists a constant  $c<\infty$ and a function $C(M)>0$ going to infinity like $M^{D-1}$ if $D$ is the degree of $P$, so that
$$
\P( |\E[\lambda_{1}(\bar X^{(N)}(N))]-\lambda_{1}(\bar X^{(N)}(N))|\ge \eta)\le 2e^{-\frac{N\eta^{2}}{4 C(M)^2}}+e^{-cN(M-M_0)}\, ,
$$
where the last term comes from bounding the matrices by $M$ to make the polynomial Lipschitz. We deduce that $\lambda_1(\bar X^{(N)}(N))$ converges almost surely towards $r_P$ and therefore $S_\delta(\bar X^{(N)}(N)):=\{
\Spec(\bar X^{(N)}(N))\cap [r_P+\delta,+\infty[=\emptyset\}$ and $S_\delta(X^{(N)}(N))=\{\Spec( X^{(N)}(N))\cap [r_P+\delta,+\infty[= \emptyset\}$ are  almost surely realized for $N$ large enough for every $\delta>0$. 
We deduce that for $N\gg \ln \delta^{-1}$ large enough,
\begin{equation}\label{bbQ}Q_{N}:=\P\left(
S_\delta(X^{(N)}(N))
\big| \Spec(X^{(N)}(0))\cap [x-\delta,x+\delta]\neq\emptyset
, \|W^{(N)}\|_{\rm{op}}\le N\right)\ge 1/2{.}\end{equation}
Indeed, \eqref{dis} implies that $\|X^{(N)}(N)-\bar X^{(N)}(N)\|_{\rm{op}}$ is of order $(cN^D+C)e^{-N/2}$ when  we are on the set  where the $W_{\ell}^{(N)}$ are bounded in operator norm by $N$ and   we are on the set $\Omega_N$ where the operator norms of $ G_N^\ell(N)$  and $\bar G_N^\ell(N)$ are all  bounded by $N$ (here $c,D,C$ depend on $P$).  For $N$ large enough so that $(N^D+C) e^{-N/2}\le \delta/2$, we deduce that $Q_N$ is bounded below by
\begin{equation*}
\begin{split}
&\ge\frac{ \P(S_{\delta/2}(\bar X^{(N)}(N))
\cap\{ \Spec(X^{(N)}(0))\cap [x-\delta,x+\delta]\neq\emptyset\}\cap  \|W^{(N)}\|_{\rm{op}}\le N\}\cap\Omega_N)}{ \P( \Spec(X^{(N)}(0))\cap [x-\delta,x+\delta]\neq\emptyset , \|W^{(N)}\|_{\rm{op}}\le N\})}\\
&\ge \frac{ \P(S_{\delta/2}(\bar X^{(N)}(N))
\P( \{ \Spec(X^{(N)}(0))\cap [x-\delta,x+\delta]\neq\emptyset\}\cap  \|W^{(N)}\|_{\rm{op}}\le N\}) 
-\P(\Omega_N^c)}{ \P(
\Spec(X^{(N)}(0))\cap [x-\delta,x+\delta]\neq\emptyset
, \|W^{(N)}\|_{\rm{op}}\le N)}\\
%&=\P(\{ |\lambda_{1}( \bar X^{(N)}(N))-{\ell}_{\nu,t}^{\Lambda}|\le \delta/2\}){.}
\end{split}
\end{equation*}
where we finally  used that $\bar X^{(N)}$ is independent from $X^{(N)}(0)$ and $W_i^{(N)}$. Finally, $P(\Omega_N^c)$ is of order $e^{-cN^2}$ whereas the denominator is of order $e^{-cN}$ since we assumed its entropy finite. Hence the last term in the above right hand side is at most of order $1/4$ for $N$ large enough. Thus,  we arrive at
\begin{equation*}
Q_{N}\ge \P(S_{\delta/2}(\bar X^{(N)}(N)))-\frac{1}{4}\, .
\end{equation*} 
Since the first term goes to one by the previous considerations, we deduce that for $N$ large enough, $Q_N$ is greater than $1/2$, namely \eqref{bbQ}.

As a consequence, we deduce 
 that the above  {right-hand} side goes to {$1$} as $N$ goes to infinity.
The fact that $Q_{N}$ is bounded {from} below by $1/2$  implies that
\begin{align}
& \P(\Spec{ X^{(N)}(0)}\cap [x-\delta,x+\delta]\neq\emptyset
, \|W^{({N})}\|_{\rm{op}}\le N)\nonumber\\
&\le 2  \P(S_{\delta/2}(X^{(N)}(N))\cap\{ \Spec{ X^{(N)}(0)}\cap [x-\delta,x+\delta]\neq \emptyset\}\cap\{ \|W^{({N})}\|_{\rm{op}}\le N\})\nonumber\\
 &\le 2  \P(S_{\delta/2}(X^{(N)}(N))\cap \{ \Spec{ X^{(N)}(0)}\cap [x-\delta,x+\delta]\neq\emptyset\}, \{\|W^{({N})}\|_{\rm{op}}\le N\}\cap\Omega'_{N })\nonumber\\
 &\qquad+2\P((\Omega_{N}')^{c})\, ,\label{bb1}
\end{align}
where
$${\Omega'_{N}=\bigcap_{0\le p\le N^{2}-1}
 \Bigr\{\Bigr\|X^{(N)}\Bigr(\frac{p+1}{N}\Bigr)-X^{(N)}\Bigr(\frac{p}{N}\Bigr)\Bigr\|_{\rm{op}}\le\delta\Bigr\}\, .} $$
Since
$G^i_{N}(\frac{p+1}{N})-G^i_{N}(\frac{p}{N})$
is a Gaussian matrix with variance $2(1-e^{-1/(2N)})\sim \frac1N$, we get 
$$\P(\max_{ i\in\{1,\ldots,k\}\atop  p\in\{1,\ldots,N^2\}}
\Big\|G_i^{(N)}\Bigr(\frac{p+1}{N}\Bigr)-G_i^{(N)}\Bigr(\frac{p}{N}\Bigr)\Big\|_{\rm{op}}\ge\delta)\le 2kN^{2}e^{-\frac{1}{4}N^{2}\delta^{2}}\, .$$
Similarly
$$\P(\max_{ i\in\{1,\ldots,k\}\atop  p\in\{1,\ldots,N^2\}}
\Big\|G_i^{(N)}\Bigr(\frac{p+1}{N}\Bigr)\Big\|_{\rm{op}}\ge N^{1/D})\le 2kN^{2}e^{-\frac{1}{4}N^{1+\frac{2}{D}}}\, .$$
Hence, taking $D$ equal to twice  the degree of $P$, we find that
\begin{eqnarray*}
&&\P(\Big\|X^{(N)}\Bigr(\frac{p+1}{N}\Bigr)-X^{(N)}\Bigr(\frac{p}{N}\Bigr)\Bigr\|_{\rm{op}}\ge\delta)\\
&\le& \P( (cN^{\frac{D-1}{2D}}+c')
\max_i\Big\|G_i^{(N)}\Bigr(\frac{p+1}{N}\Bigr)-G_i^{(N)}\Bigr(\frac{p}{N}\Bigr)\Big\|_{\rm{op}}\ge\delta)+2kN^{2}e^{-\frac{1}{4}N^{1+\frac{2}{D}}}\\
&\le &
2kN^{2}e^{-\frac{1}{4}(\frac{N}{cN^{\frac{D-1}{2D}}+c'})^2\delta^{2}}+2kN^{2}e^{-\frac{1}{4}N^{1+\frac{2}{D}}}\, .
\end{eqnarray*}
We deduce that $\P((\Omega_N')^c)$ is at most of order $e^{-c N^{1+\frac{1}{2D}}\delta^2}$ for some positive constant $c$.

Moreover, we claim that the spectrum $\Spec(X)$ is continuous. In fact,
if $\lambda_N(X)\le \lambda_{N-1}(X)\le \cdots \lambda_1(X)$, the distance between the spectrum of two self-adjoint $N\times N$ matrices $X,Y$ satisfies
$$\max_{1\le i\le N}|\lambda_i(X)-\lambda_i(Y)|\le \|X-Y\|_{\rm{op}}$$
by Weyl's inequalities\cite[Theorem A.5]{AGZBook}. As a consequence, if the $i$-th eigenvalue of $X^{(N)}(0)$ belongs to $[x-\delta,x+\delta]$, since $y$ lies in between $x$ and $r_P$, we see that on $\Omega_N'\cap S_{\delta/2}(X^{(N)}(N))$ there
must be a $p\in [0,N^2]$ so that $\lambda_i(X^{(N)}(\frac{p}{N})\in [y-\delta,y+\delta]$. Therefore,
\begin{multline*}
\{S_{\delta/2}(X^{(N)}(N))\cap \{ \Spec{ X^{(N)}(0)}\cap [x-\delta,x+\delta]\neq\emptyset\}, \{\|W^{({N})}\|_{\rm{op}}\le N\}\cap\Omega'_{N }\}
\\
\subset {\bigcup_{0\le p\le  N^{2}}\Bigl\{
\Spec(X^{(N)}(\frac{p}{N}))\cap[y-\delta,y+\delta]\neq \emptyset\Bigr\}.}
\end{multline*}
Hence, we conclude from \eqref{bb1} that
\begin{equation*}
\begin{split}
 &\P( \{ \Spec{ X^{(N)}(0)}\cap [x-\delta,x+\delta]\neq\emptyset\}
 \cap  \|W^{(N)}\|_{\rm{op}}\le N)\\
 &\leq 2 \P\Bigl(\bigl( \bigcup_{0\le p\le N^{2}}
 \Bigl\{\Spec(X^{(N)}(\frac{p}{N}))\cap[y-\delta,y+\delta]\neq \emptyset\bigr\}\Bigr)+e^{-c N^{1+\frac{1}{2D}}\delta^2}
\\
&\leq 2N^{2 }\P(\Spec(P_N)\cap[y-\delta,y+\delta]\neq \emptyset)
+e^{-c N^{1+\frac{1}{2D}}\delta^2},
\end{split}
\end{equation*}
where we used that $X^{(N)}(\frac{p}{N})$ has the same law as $X^{(N)}(0)$. 
Finally, since again by concentration of measure $\P( \max_i\|W_{i}^{(N)}\|_{\rm{op}}\ge N)\le e^{-cN^{2}}$,
we conclude that for some $c'(\delta)>0$ and all $N\in\N$,
\begin{equation*}
{ \P( \Spec(P_N)\cap[x-\delta,x+\delta]\neq \emptyset
)\leq 2N^{2 }\P(\Spec(P_N)\cap[y-\delta,y+\delta]\neq \emptyset)+e^{-c N^{1+\frac{1}{2D}}\delta^2}.}
\end{equation*}
The monotonicity of $\mathcal{I}_{\Spec(P)}$ follows from \eqref{wldpspec} as long as $ \P (\Spec(P_N)\cap[x-\delta,x+\delta]\neq \emptyset)$ is much larger than $e^{-cN^{2}}$ which is clear since we assumed that $\mathcal I_{\Spec(P)}(x)$ is finite.  Therefore $\inf_{[x,+\infty[}\mathcal I_{\Spec(P)}=\mathcal I_{\Spec(P)}(x)$.

\medskip
We finally  need to prove that for every $\varepsilon >0$,
    $$ \limsup_{N\to\infty}\frac{1}{N}\ln \P(\lambda_1(P)\ge x)=\limsup_{N\to\infty}\frac{1}{N}\ln \P(\lambda_1(P)\in [ x, x+\varepsilon])\, .$$ 
    
For this, one only needs to prove that for every $x,y$ such that $r_{P}< x< y$ and $\delta >0$, we have 

\begin{equation*}
    \limsup_{N \to \infty}\frac{1}{N}\ln \P( \lambda_1(P) \in ] x - \delta, x + \delta[) \geq \limsup_{N \to \infty} \frac{1}{N}\ln\P( \lambda_1(P) \in ] y - \delta, y + \delta[)
\end{equation*}
then to use a finite covering of $[x,M]$ by $M/\varepsilon$ segments of size $\varepsilon$ and the control of $\P(\lambda_1(P)\ge M)$ by \eqref{exptigl}.
    A similar result holds for the liminf, so that \eqref{wldub} and \eqref{wldlb} imply the standard weak  large deviation principle. Together with exponential tightness, this concludes the proof of Theorem \ref{thm-main}(ii) with $\mathcal I_P(x)=\mathcal{I}_{\Spec(P)}(x)$. 
For this, one can actually use the same argument as for Part (i) of the theorem. One only has to replace all the events of the form $\{\Spec(X)\cap [x- \delta,x+ \delta]\neq\emptyset\}$ and $\{\Spec(X)\cap [y- \delta,y+ \delta]\neq \emptyset\}$ with $X$ a random matrix by respectively $\lambda_1(X)\in [x- \delta,x+ \delta]$ and $\lambda_1(X)\in [y- \delta,y+ \delta]$. Since the function $\lambda_1$ by itself is $1$-Lipschitz with regard to the operator norm, the crux of the proof remains unchanged and one gets : 

\begin{equation*}
\P( \lambda_1(P_N) \in [x-\delta,x+\delta]
)\leq 2N^{2 }\P(\lambda_1(P_N)\in [y-\delta,y+\delta])+e^{-c N^{1+\frac{1}{2D}}\delta^2}.
\end{equation*}

This statement is actually closer to the original Lemma in \cite{BoursierG24}.
Taking the $\log$, dividing by $N$ and then taking the $\liminf$ and $\limsup$ gives us the desired result.
The same arguments show  the large deviation principle for the smallest eigenvalue $\lambda_N(P_N)$ (or simply change $P$ into $-P$).

\medskip
\underline{\smash{Properties of the rate function}}: We  have proven that $\mathcal I_P(x)= \mathcal I_{\Spec(P)}(x)=\mathcal{I}_{L_x}$ is the good rate function for the  large deviation principle of $\lambda_1(P)$ on $(r_P,+\infty)$. We have seen it is increasing on $]r_P,+\infty[$.  Finally, by \eqref{exptigl} and \eqref{exptih}, we see that for every $\delta>0$
    $$\limsup_{N\to\infty }\frac{1}{N}\ln \P(\lambda_1(P)>r_P+\delta)<0$$
    which implies that $\mathcal{I}_{\lambda_{\rm{max}}(P)}(x)>0$ if $x>r_P$. Of course $\mathcal{I}_{\lambda_{\rm{max}}(P)}(r_P)=0$ as $\P(\lambda_1(P)\in [r_P-\delta, r_P+\delta])$ goes to one for every $\delta$. Similarly, \eqref{convd} implies that $\mathcal{I}_{\lambda_{\rm{max}}(P)}(x)$ is infinite if $x<r_P$. It is also clear that $\mathcal{I}_{\Spec(P)}$ vanishes on the support of $\mu_P$.

\medskip
Lastly, to obtain the large deviation principle for the operator norm $\|P_N\|_{\rm{op}}$ in Part (iii) of Theorem~\ref{thm-main}, it is enough to notice that
$\|P_N\|_{\rm{op}}=\lambda_1(P_NP_N^*)^{1/2}$ with $P_NP_N^*=(PP^*)_N$ the  polynomial $PP^*$ evaluated at
$(W_1^{(N)},\ldots, W_k^{(N)})$. This gives the LDP with $\mathcal{I}_{\|P\|_{\rm{op}}}(x) =\mathcal{I}_{\lambda_{\rm{max}}(PP^*)}(x^2)$. In particular $\mathcal{I}_{\|P\|_{\rm{op}}}$ is increasing. We can get the alternative formula $\mathcal{I}_{\|P\|_{\rm{op}}}(x)=\min\{\mathcal{I}_{\Spec(P)}(x),\mathcal{I}_{\Spec(P)}(-x)\}$ by noticing that
$$\{\|P_N\|_{\rm{op}}\ge x\}
=\{\lambda_1(P_N)\ge x\}\cup
\{\lambda_N(P_N)\le -x\}=\{\lambda_1(P_N)\ge x\}\cup
\{\lambda_1(-P_N)\ge x\}$$
which implies 
$$\inf_{z\ge x}\mathcal{I}_{\|P\|_{\rm{op}}}(z)=\min\{\inf_{z\ge x}\mathcal{I}_{\lambda_{\rm{max}}(P)}(z), \inf_{z\ge x}\mathcal{I}_{\lambda_{\max}(-P)}(z)\}\,.$$
Using the monotonicity of all the involved rate functions and their relations with $\mathcal{I}_{\Spec(P)}$ yields the result.

\end{proof}
\section{Proof of Corollary \ref{cor-multiP}}\label{sec-multiP}
To simplify the notations, we consider the case $\beta=1$, but the proof generalizes to the case $\beta=2$ right away.
The exponential tightness follows readily from the one-polynomial case. 
Hence, the main issue is to prove the weak large deviation principle, namely to estimate the probability of small balls such as
$$\mathcal{B}_\varepsilon(x_1,\ldots,x_r):=
\bigcap_{1\le i\le r}\{ \Spec(P_i(W_1^{(N)},\ldots, W_k^{(N)})
)\cap [x_i-\varepsilon,x_i+\varepsilon]\neq \emptyset\}\,.$$
As in \eqref{eq-linUB} and \eqref{eq-linLB}, we easily see that this is estimating the probability of small singular values of their linearizations
$$\mathcal{B}^L_\delta(x_1,\ldots,x_r):=
\bigcap_{1\le i\le r}\{ \lambda_{N\ell_i}(L_{x_i}^{P_i}(L_{x_i}^{P_i})^*)\le \delta
\}\, ,$$
where $L_{x_i}^{P_i}$ corresponds to a linearization of the polynomial $P_i$ as in \eqref{eq-defLz} with dimension $\ell_i$.  We remark that this is equivalent to considering the block diagonal linearization $L_{x_1,\ldots,x_r}:=\diag(L_{x_1}^{P_1},\ldots,L_{x_r}^{P_r})$
and optimize on vectors with block-wise constant norms:
$$\lambda^{\rm{min}}_{x_1,\ldots,x_r}:=
\min_{u=(u^1,\ldots, u^r)\atop u^i\in \mathds{S}^{N\ell_i-1}, 1\le i\le r}\|L_{x_1,\ldots,x_r} u\|_2^2=\sum_{i=1}^r
\lambda_{N\ell_i}(L_{x_i}^{P_i}(L_{x_i}^{P_i})^*)
$$
is small. Hence, most computations are the same except we have to enforce the restriction of the vector $u$ to the span $V_s$ of the vectors $e_i$ of the canonical basis with index $i\in [1+\sum_{p\le s-1} N\ell_p,\sum_{p\le s} N\ell_p ]$ to belong to the unit sphere rather than the vector itself. This implies that the matrix
$$\rho_r(u)= \langle u_i^s,u_j^t\rangle$$
with $s,t\in \{1,\ldots, r\}, i\le \ell_s, j\le \ell_t$ belongs to the space
$\mathrm{Sym}_{\ell_1,\ldots,\ell_r}^{+,1}(\R)$ of matrices in 
$\mathrm{Sym}^+_{\sum \ell_i}(\R)$ whose restriction to $V_i=\{ e_s,s\in [\sum_{j\le i-1}\ell_j +1, \sum_{j\le i} \ell_j]$ has trace one for every $k\in\{1,\ldots,r\}$. Here, $(e_i)_{1\le i\le \sum \ell_k}$ denotes the canonical basis.

\medskip
We will only detail the arguments to get the weak large deviation principle, since exponential tightness readily follows from the one-polynomial case. We infer the following weak large deviation principle.
\begin{lemma} For every $x_1,\ldots,x_r$,
\begin{eqnarray*}
-\mathcal{I}_{L^{P_1}_{x_1},\ldots, L^{P_r}_{x_r}}&:=&\limsup_{\varepsilon\to 0}\limsup_{N\to\infty}\frac{1}{N}\log \P(
\lambda^{\rm{min}}_{x_1,\ldots,x_r}\le \varepsilon)\\
&=&\liminf_{\varepsilon\to 0}\liminf_{N\to\infty}\frac{1}{N}\log \P(
\lambda^{\rm{min}}_{x_1,\ldots,x_r}\le \varepsilon)\,.\end{eqnarray*}
Moreover, with the previous notations, 
$$-\mathcal{I}_{L^{P_1}_{x_1},\ldots, L^{P_r}_{x_r}}=
\lim_{\theta\to\infty}\sup_{\Psi } [F(\theta,\Psi)-\widetilde{\mathcal{I}}_r(\Psi)-\sum_{i=1}^r\ell_i J_{L_{x_i}^{P_i}}(\theta)]\, ,$$
where
$\widetilde{\mathcal{I}}_r(\Psi)$ is infinite unless $\Psi$ belongs to $\mathrm{Sym}_{\ell_1,\ldots,\ell_r}^{+,1}(\R)$ and then
$$\widetilde{\mathcal{I}}_r(\Psi)=-\frac{1}{2}
\{\ln\det(\Psi)+\sum \ell_i \ln \ell_i
\}\,.$$
\end{lemma}
\begin{proof}
The proof of the upper bound follows from the tilting argument:
$$\P(\lambda^{\rm{min}}_{x_1,\ldots,x_r}\le \varepsilon)\le e^{- N\sum_{i=1}^r (\ell_i J_{L^{P_i}_{x_i}}(\theta)+o(1))}\E_u \E_W[e^{-\frac{\theta}{2} N\sum_{i=1}^r \ell_i\|L_{x_i}^{P_i}u_i\|_2^2}]\, .$$
The computation of $$\E_W[e^{-\frac{\theta}{2} N\sum_{i=1}^r\ell_i \|L_{x_i}^{P_i}u_i\|_2^2}]
=\E_W[e^{-\frac{\theta}{2} N \|D_\ell^{1/2} L_{x_1,\ldots,x_r}u\|_2^2}]$$
with $D_\ell=\mathrm{diag}(\ell_1 \Id_{N\ell_1},\ldots, \ell_r\Id_{N\ell_r})$ and  $u=(u_1,\ldots, u_r)$ is the same as in Lemma \ref{lem-annealedcalc}. The proof of the large deviation principle for $\rho_r(u)$ is also a straightforward extension of Lemma \ref{lem-WishartLDP} as $\rho(u)$ is a Wishart matrix conditioned to belong to $\mathrm{Sym}_{\ell_1,\ldots,\ell_r}^{+,1}(\R)$: hence, its distribution is known and the result follows from classical Laplace method.

\medskip
To prove the lower bound, we first introduce the analogue of $\sigma_0^N$ in Lemma \ref{lem-claim1} where $A_0,\ldots,A_k$ are such that $L_{x_1,\ldots, x_r}=\sum A_k\otimes X_k+A_0\otimes \rm{Id}$ and the smallest singular value is replaced by 
$$\sigma_{0,r}^N=\min_{u=(u_1,\ldots,u_r)\atop u_i \in\mathds{S}^{N\ell_i-1}, 1\le i\le r}\min_{M_1,\ldots, M_k\in \mathrm{Sym}_N(\R)}\|(\sum A_k\otimes M_k+A_0\otimes I)u\|_2^2\, .$$
As in Lemma \ref{lem-claim1} we readily show that $\sigma_{0,r}^N$ does not depend on $N$ and as in Lemma \ref{lem-claim2} that $\mathcal{I}_{x_1,\ldots,x_r}$ is infinite when $\sigma_{0,r}>0$. When $\sigma_{0,r}=0$, we can show as in Lemma \ref{lem-claim3} that the annealed spherical integral is of order $e^{o(1)\theta N}$, from which we can see that under the tilted measure the spectrum of $L_{x_1,\ldots,x_r}$ has the desired deviation.  Hence, the proof is the same as in the one polynomial case.
\end{proof}

\section{Modifications for the GUE ($\beta=2$)}\label{sect-betas}
While the overall argument to prove Theorem~\ref{thm-main} and Corollary~\ref{cor-multiP} is analogous for the case $\beta=2$, there are some small modifications, especially in the covariance computation in Section~\ref{sect-UB}, that are worth pointing out. We collect them in this section.

\subsection{Upper Bound in Theorem~\ref{thm-Lbounds}}
Apart from the obvious replacements of tilting by $I_{2,N\ell}(-\theta,L_zL_z^*)$ instead of $I_{1,N\ell}(-\theta,L_zL_z^*)$ and replacing the real inner product in the definition of the profile $\rho(u)$ by its complex analog, the main change lies in the computation of the annealed spherical integral in Lemma~\ref{lem-annealedcalc}. For the GUE, the ``transposed'' term in~\eqref{eq-cov-correct} no longer appears in the covariance of $L_z^{(N)}u$ but in its pseudo-covariance. 

\medskip
To formulate the analog of Lemma~\ref{lem-annealedcalc}, we consider $\C^{\ell\times\ell}$ as a real vector space of dimension $2\ell^2$, endowed with the inner product $\langle Y,Z\rangle_{\R}:=\re\Tr(Y^*Z)$, and $\det_\R$ denotes the determinant of an $\R$-linear map. For $\Psi\in\mathrm{Sym}^{+,1}_\ell(\C)$ we set $C_\Psi:=\sum_{j=1}^kA_j\Psi A_j^*$ and $B=\Id_\ell+\theta\ell\,C_\Psi$ as before, and define
\begin{align*}
\mathcal{R}^{(2)}_\Psi(Z):=\sum_{j=1}^kA_j\Psi^{1/2}Z^*A_j\Psi^{1/2}\, ,\quad \mathcal{K}^{(2)}_\Psi(Z):=C_\Psi Z+\mathcal{R}^{(2)}_\Psi(Z)\, .
\end{align*}
Note that $\mathcal{R}^{(2)}_\Psi$ and $\mathcal{K}^{(2)}_\Psi$ are $\R$-linear but not $\C$-linear.

\begin{lemma}\label{lem-annealedcalc-GUE}
Let $\beta=2$. Under Assumption~\ref{assum-main}, for every $u$ in the unit sphere of $\C^{N\ell}$,
\begin{displaymath}
\E_W[\exp(-\theta N\ell\|L_z^{(N)}u\|^2)]=\exp(f^{(2)}_N(\theta,\rho(u)))\, ,
\end{displaymath}
where
\begin{displaymath}
f^{(2)}_N(\theta,\Psi)=-N\theta\ell\,Q^{(2)}_\theta(\Psi)-N\ln(\det(B))-\frac12\ln{\det}_\R\big(\Id_{2\ell^2}+\theta\ell B^{-1}\mathcal{R}^{(2)}_\Psi\big)\, ,
\end{displaymath}
the map $\mathcal{K}^{(2)}_\Psi$ is self-adjoint and non-negative for $\langle\cdot,\cdot\rangle_\R$, and
\begin{displaymath}
Q^{(2)}_\theta(\Psi):=\big\langle A_0\Psi^{1/2},(\Id_{2\ell^2}+\theta\ell\,\mathcal{K}^{(2)}_\Psi)^{-1}(A_0\Psi^{1/2})\big\rangle_\R\, .
\end{displaymath}

\end{lemma}

\begin{proof}
Let $\xi=(A_0\otimes\Id_N)u$ and $G=L_z^{(N)}u-\xi=\sum_j(A_j\otimes W_j)u$. For a GUE matrix $W$ and deterministic $x,y\in\C^N$, we have
\begin{displaymath}
\E\big[(Wx)(Wy)^*\big]=\frac1N\,(y^*x)\,\Id_N,\qquad \E\big[(Wx)(Wy)^T\big]=\frac1N\,y\,x^T .
\end{displaymath}
Setting $v^{(j)}_a=\sum_b(A_j)_{ab}u_b$ as before, we deduce that $\Gamma:=N\E[GG^*]=C_\Psi\otimes\Id_N$, while the pseudo-covariance $\Pi:=N\E[GG^T]$ can be decomposed into $N\times N$ blocks of the form
\begin{displaymath}
\Pi_{aa'}=\sum_{j=1}^kv^{(j)}_{a'}\big(v^{(j)}_a\big)^T\, .
\end{displaymath}
We see $G$ as a real Gaussian vector of $\C^{N\ell}\simeq\R^{2N\ell}$, with the scalar product $\re(x^*y)$. Since $\E[(\re w^*G)^2]=\frac{1}{2N}\re\big(w^*\Gamma w+w^*\Pi\bar w\big)$ and $\Pi$ is symmetric, its covariance is the $\R$-linear operator $\frac{1}{2N}\mathcal{S}$, where $\mathcal{S}w:=\Gamma w+\Pi\bar w$. The real Gaussian integration formula used in the proof of Lemma~\ref{lem-annealedcalc}, applied with $\theta N\ell$ in place of $\theta$, gives
\begin{equation}\label{eq-expectationLu-GUE}
\E_W[\exp(-\theta N\ell\|L_z^{(N)}u\|^2)]=\frac{\exp\big(-\theta N\ell\,\re\langle\xi,(\Id+\theta\ell\mathcal{V}_u)^{-1}\xi\rangle\big)}{\sqrt{{\det}_\R(\Id+\theta\ell\mathcal{S})}} .
\end{equation}
Assuming that $\Psi=\rho(u)$ is invertible, let $\mathcal{V}_u:=\mathrm{span}_\C\{e_a\otimes u_b\}$ and $\mathcal{U}(Z):=((Z\Psi^{-1/2})\otimes\Id_N)u$. Since $((X\otimes\Id_N)u)^*(Y\otimes\Id_N)u=\Tr(Y\Psi X^*)$, $\mathcal{U}$ is an isometry from $(\C^{\ell\times\ell},\langle\cdot,\cdot\rangle_\R)$ onto $\mathcal{V}_u$. As in the proof of Lemma~\ref{lem-annealedcalc}, we have $\xi=\mathcal{U}(A_0\Psi^{1/2})$ and $\Gamma\,\mathcal{U}(Z)=\mathcal{U}(C_\Psi Z)$. Moreover, if $x=\mathcal{U}(Z)$ and $X=Z\Psi^{-1/2}$, then $u_b^T\bar u_d=\Psi_{bd}$ gives $u_b^T\bar x_{a'}=(\Psi X^*)_{ba'}$. Hence,
\begin{displaymath}
(\Pi\bar x)_a=\sum_{j=1}^k\sum_{a'}v^{(j)}_{a'}(A_j\Psi X^*)_{aa'}=\sum_{j=1}^k\sum_d(A_j\Psi X^*A_j)_{ad}\,u_d ,
\end{displaymath}
i.e., $\Pi\bar x=\mathcal{U}(\mathcal{R}^{(2)}_\Psi(Z))$. If $x\perp \mathcal{V}_u$, then $u_b^T\bar x_{a'}=\overline{u_b^*x_{a'}}=0$ and $\Pi\bar x=0$. Therefore, $\mathcal{V}_u$ and $\mathcal{V}_u^\perp$ are invariant under $\mathcal{S}$, $\mathcal{S}=C_\Psi\otimes\Id_N$ on $\mathcal{V}_u^\perp$, and $\mathcal{S}$ is conjugated to $\mathcal{K}^{(2)}_\Psi$ on $\mathcal{V}_u$. In particular $\mathcal{K}^{(2)}_\Psi$ is self-adjoint and non-negative and $\re\langle\xi,(\Id+\theta\ell\mathcal{S})^{-1}\xi\rangle=Q^{(2)}_\theta(\Psi)$.

\medskip
Lastly, we consider the determinant in the denominator of~\eqref{eq-expectationLu-GUE}. Recall that the real determinant of a $\C$-linear map is the squared modulus of its complex determinant. Hence, ${\det}_\R(\Id+\theta\ell\,C_\Psi\otimes\Id_N)=\det(B)^{2N}$, and the restriction of this map to $\mathcal{V}_u$ is conjugated to $Z\mapsto BZ$, of real determinant $\det(B)^{2\ell}$. Since $\Id+\theta\ell\mathcal{K}^{(2)}_\Psi=B(\Id+\theta\ell B^{-1}\mathcal{R}^{(2)}_\Psi)$ on $\C^{\ell\times\ell}$, we get
\begin{displaymath}
{\det}_\R(\Id+\theta\ell\mathcal{S})=\det(B)^{2N}\,{\det}_\R\big(\Id_{2\ell^2}+\theta\ell B^{-1}\mathcal{R}^{(2)}_\Psi\big).
\end{displaymath}
Inserting the above identities into~\eqref{eq-expectationLu-GUE} proves the lemma for invertible $\Psi$. The general case follows by continuity and density exactly as in the proof of Lemma~\ref{lem-annealedcalc}. 
\end{proof}

Consequently, writing $f^{(2)}_N=2NF^{(2)}+\widetilde{f}^{(2)}_N$ with
\begin{equation}\label{eq-defF-GUE}
F^{(2)}(\theta,\Psi)=-\frac{1}{2}\theta\ell\,Q^{(2)}_\theta(\Psi)-\frac{1}{2}\ln\det(\Id_\ell+\theta\ell\,C_\Psi),
\end{equation}
we have $\sup_\Psi|\widetilde f^{(2)}_N(\theta,\Psi)|\leq\ell^2\ln(1+2\theta\ell\sum_j\|A_j\|^2)$, by the same argument as in the GOE case.

\medskip
As in Section~\ref{sect-UB}, the convergence of the annealed spherical integral now follows from Varadhan's lemma, using Lemma~\ref{lem-annealedcalc-GUE} and Lemma~\ref{lem-WishartLDP} as input.

\begin{remark}
In the real case, the same notation reads $F^{(1)}(\theta,\Psi)=\frac12\big(-\theta\ell Q_\theta-\ln\det(\Id+\theta\ell C_\Psi)\big)$ with $\mathcal{R}_\Psi(Z)=\sum_jA_j\Psi^{1/2}Z^TA_j\Psi^{1/2}$. Hence, for $\beta\in\{1,2\}$, $$F^{(\beta)}(\theta,\Psi)=\big(-\frac{\theta}{2}\ell Q^{(\beta)}_\theta-\frac{1}{2}\ln\det(\Id+\theta\ell C_\Psi)\big),$$ where $Z^T$ is replaced by $Z^*$ when $\beta=2$. For $P=X$ (i.e., $\ell=k=1$, $A_0=z$, $A_1=-1$), one finds $\mathcal{K}^{(2)}_1(Z)=2\re Z$ and $Q^{(2)}_\theta(1)=Q_\theta(1)=z^2/(1+2\theta)$.
\end{remark}

\subsection{Lower Bound in Theorem~\ref{thm-Lbounds}}
As for the upper bound, we consider $L_z^{(N)}u$ as a real Gaussian vector of $\R^{2N\ell}$ with covariance $\frac1{2N}\mathcal{S}$. When applying Lemma~\ref{lem-tiltedGaussian}, the analog of Equation~\eqref{eq-tiltedex} reads
\begin{displaymath}
\E^{(u,\theta)}[\|L_z^{(N)}u\|_2^2]=\frac{1}{2N}\Tr_\R[\mathcal{S}(\Id+\theta\ell\mathcal{S})^{-1}]+\|(\Id+\theta\ell\mathcal{S})^{-1}\xi\|_2^2\, .
\end{displaymath}
For $A_0=0$, the argument holds unchanged, and we only note some minor modifications for the general case. 

\medskip
The analog of the constant $\sigma_0$ is given by
\begin{displaymath}
\sigma^N_0 := \min_{M_1,\dots,M_{k} \in\mathrm{Sym}_N(\C)} \Bigl\{ \sigma_{\min}\Bigl( \sum_{k} A_k \otimes M_k + A_0 \otimes \Id_N\Bigr)^2 \Bigr\}\, ,
\end{displaymath}
from which rest of the proof of Lemma~\ref{lem-claim2} is readily adapted. In particular, $2J_{\nu_z}(\theta)$ still grows like $\ln\theta$, such that the conclusion of the lemma is unchanged. 

\medskip
For Lemma~\ref{lem-claim3}, we replace any properties of the GOE used in the proof by the corresponding properties of the GUE: The $\ell\times\ell$ top-left corner of a GUE matrix has density proportional to $\exp(-\frac N2\Tr(H^2))$ on $\mathrm{Sym}_\ell(\C)$ and the trace $\Tr(Y_j^N(Y_j^N)^*)$ is a Gamma variable with shape parameter $(N-\ell)\ell$ and scale parameter $1/N$, which also satisfies a large deviation principle with speed $N$, but the rate function is given by $x\mapsto x-\ell-\ell\ln(x/\ell)$ here. Similarly, an estimate of the form~\eqref{eq-uminusuhat} holds for $u$ uniformly distributed on the complex unit sphere if we double the exponent on the right-hand side to account for the change in dimension. While these modifications change the value of the constant $c(\epsilon)$, the claim of the lemma holds unchanged.

\medskip
With Theorem~\ref{thm-Lbounds} adapted to the case $\beta=2$, the argument to prove the main results is essentially unchanged.

\appendix
\section{Additional Proofs}\label{sect-app}

\subsection{Proof of Lemma~\ref{lem-conc}}
We start by noting that both the existence of the limit $\nu_z$ and the concentration result are immediate from~\cite{LDP_Kronecker} whenever $A_0,\dots,A_k$ are self-adjoint. Indeed, let us denote
\begin{displaymath}
\mu_H(\Pi_i,\Pi_j)=\frac{1}{N}\sum_{a=1}^{N\ell}\langle \Pi_iv_a(H),\Pi_jv_a(H)\rangle\delta_{\lambda_a(H)}
\end{displaymath}
for a self-adjoint matrix $H\in\K_\beta^{N\ell\times N\ell}$ with eigenvectors $v_a(H)$ and eigenvalues $\lambda_a(H)$. Here, the projections $\Pi_i$ and $\Pi_j$ map the vector $v_a(H)\in\K_\beta^{N\ell}$ to its $i$-th and $j$-th block of length $N$, respectively. By setting
\begin{displaymath}
\frac{1}{\ell}\sum_{j=1}^\ell\mu_H(\Pi_j,\Pi_j)=\frac{1}{N\ell}\sum_{a=1}^{N\ell}\delta_{\lambda_a(H)}\, ,
\end{displaymath}
we obtain the empirical eigenvalue measure, and the desired result follows from~\cite[Thm.~4.4]{LDP_Kronecker}, since
\begin{align*}
\P(d_{KR}(\mu_N,\nu_z)\geq\epsilon)&\leq\P(\exists j\in\{1,\dots,\ell\}:d_{KR}(\mu_{L_z^{(N)}(L_z^{(N)})^*}(\Pi_j,\Pi_j),\mu_{j,j})\geq\epsilon)\\
&\leq\P(\exists i,j\in\{1,\dots,\ell\}:d_{KR}(\mu_{L_z^{(N)}(L_z^{(N)})^*}(\Pi_i,\Pi_j),\mu_{i,j})\geq\epsilon)\, ,
\end{align*}
where $d_{KR}$ denotes the Kantorovich-Rubinstein distance and we follow the notation in~\cite[Sect.~1.1]{LDP_Kronecker} for the limits $\mu_{i,j}$. Recall that
\begin{displaymath}
d_{KR}(\nu_1,\nu_2)=\sup_{\substack{\|f\|_\infty\leq1,\\ \mathrm{Lip}(f)\leq1}}\int f\dx(\nu_1-\nu_2)\, ,
\end{displaymath}
where $\mathrm{Lip}(f)$ denotes the Lipschitz constant associated to a Lipschitz function $f$.

\medskip
As the linearization $L_z^{(N)}$ is not necessarily self-adjoint, we start by looking at a suitable symmetrization. As $W_1,\dots,W_k$ are GOE or GUE matrices and in particular self-adjoint, we have that
\begin{equation}\label{eq-defsymm}
\mathcal{L}_z^{(N)}:=\begin{pmatrix}0&L_z^{(N)}\\(L_z^{(N)})^*&0\end{pmatrix}=\sum_{j=1}^k\begin{pmatrix}0&A_j\\A_j^*&0\end{pmatrix}\otimes W^{(N)}_j+\begin{pmatrix}0&A_0\\A_0^*&0\end{pmatrix}\otimes \Id_N\, .
\end{equation}
Note that~$\mathcal{L}_z^{(N)}$ is a Gaussian Kronecker matrix in the sense of~\cite{LDP_Kronecker}. Moreover, if we denote the empirical eigenvalue measure of $\mathcal{L}_z^{(N)}$ by $\mu_{\mathcal{L}_z^{(N)}}$, we have the relation
\begin{displaymath}
\int_{\R^+}f(x)\mu_{L_z^{(N)}(L_z^{(N)})^*}(\dx x)=\int_\R f(x^2)\mu_{\mathcal{L}_z^{(N)}}(\dx x)\, ,
\end{displaymath}
from which the claim of Lemma~\ref{lem-conc} follows. In particular,~\eqref{eq-defdsq} emerges as a natural distance in this setting.

\subsection{Proof of Lemma~\ref{lem-prelimbounds}}
(i) This is an immediate consequence of the above proof of Lemma~\ref{lem-conc}.

\medskip
(ii) While $L_z^{(N)}(L_z^{(N)})^*$ may not match the definition of a Gaussian Kronecker matrix in~\cite{LDP_Kronecker}, the symmetrization $\mathcal{L}_z^{(N)}$ in~\eqref{eq-defsymm} does. Hence, by~\cite[Lem.~1.2]{LDP_Kronecker} (or~\cite[Thm.~4.7]{AEKN2019}),
\begin{displaymath}
\lim_{N\to\infty}\P(|\lambda_1(\mathcal{L}_z^{(N)})-\widetilde{r}_z|\geq\delta)=0
\end{displaymath}
for any $\delta>0$, where $\widetilde{r}_z$ denotes the right endpoint of the support of the corresponding limiting measure. Noting that the largest eigenvalue of $\mathcal{L}_z^{(N)}$ is given by the largest singular value $\sigma_{\max}$ of $L_z^{(N)}$ and that the largest eigenvalue of $L_z^{(N)}(L_z^{(N)})^*$ is given by $\sigma_{\max}^2$, the convergence of the largest eigenvalue readily follows from the preservation of convergence in probability under continuous mappings. The convergence of the smallest eigenvalue follows also from
\cite[Thm.~4.7(i)]{AEKN2019} which asserts the convergence of the spectrum.

\medskip
(iii) Noting $\|L_z^{(N)}(L_z^{(N)})^*\|_{\rm{op}}=\|L_z^{(N)}\|_{\rm{op}}^2$ and $\|\mathcal{L}_z^{(N)}\|_{\rm{op}}=\|L_z^{(N)}\|_{\rm{op}}$, we obtain
\begin{displaymath}
\frac{1}{N}\ln\P(\|L_z^{(N)}(L_z^{(N)})^*\|_{\rm{op}}>\alpha)=\frac{1}{N}\ln\P(\|L_z^{(N)}\|_{\rm{op}}>\sqrt{\alpha})=\frac{1}{N}\ln\P(\|\mathcal{L}_z^{(N)}\|_{\rm{op}}>\sqrt{\alpha})\, .
\end{displaymath}
The claim now follows directly from~\cite[Lem.~4.5]{LDP_Kronecker}.

\renewcommand*{\bibname}{References}

\interlinepenalty=10000

\bibliographystyle{plain}
\bibliography{References_LDP_Poly}

\end{document}